\documentclass[aap, preprint]{imsart}

\RequirePackage[OT1]{fontenc}
\RequirePackage{amsthm,amsmath}
\usepackage[colorlinks,citecolor=blue,urlcolor=blue]{hyperref}
\usepackage[numbers, sort&compress]{natbib}

\usepackage{amsmath, amsthm, amsfonts, amssymb, amsbsy, bigstrut, graphicx, enumerate,  upref, longtable, comment, booktabs, array, caption, subcaption}

\newcommand{\I}{\mathrm{i}}

\newcommand{\me}{\mathcal{E}}

\newcommand{\ms}{\mathcal{S}}

\newtheorem{thm}{Theorem}[section]
\newtheorem{lmm}[thm]{Lemma}

\newtheorem{prop}[thm]{Proposition}
\newtheorem{defn}[thm]{Definition}

\theoremstyle{definition}

\newtheorem{ex}[thm]{Example}

\newcommand{\cc}{\mathbb{C}}

\newcommand{\ee}{\mathbb{E}}

\newcommand{\rr}{\mathbb{R}}
\newcommand{\smallavg}[1]{\langle #1 \rangle}

\newcommand{\var}{\mathrm{Var}}
\newcommand{\ve}{\varepsilon}

\newcommand{\zz}{\mathbb{Z}}

\numberwithin{equation}{section}

\renewcommand{\Re}{\operatorname{Re}}

\newcommand{\tmix}{\tau_{\textup{mix}}}
\newcommand{\trel}{\tau_{\textup{rel}}}
\newcommand{\mumin}{\mu_{\textup{min}}}

\usepackage[utf8]{inputenc}
\usepackage{amsmath}
\usepackage{amsfonts}
\usepackage{bbm}
\usepackage{mathtools}
\usepackage{tikz}
\usetikzlibrary{decorations.pathreplacing}
\usepackage{amsthm}
\usepackage[english]{babel}
\usepackage{fancyhdr}
\usepackage{amssymb}
\usepackage{enumitem}
\usepackage{qtree}
\usepackage{mathrsfs}

\newcommand{\Z}{\mathbb{Z}}

\newcommand{\R}{\mathbb{R}}

\newcommand{\C}{\mathbb{C}}

\newcommand{\E}{\mathbb{E}}

\renewcommand{\bar}{\overline}
\renewcommand{\P}{\mathbb{P}}

\renewcommand{\tilde}{\widetilde}

\begin{document}

\begin{frontmatter}
\title{Spectral gap of nonreversible Markov chains}
\runtitle{Spectral gap of nonreversible Markov chains}

\begin{aug}
\author{\fnms{Sourav} \snm{Chatterjee}\ead[label=e1]{souravc@stanford.edu}}

\runauthor{Sourav Chatterjee}

\affiliation{Stanford University}

\address{Department of Statistics\\
390 Jane Stanford Way\\
Stanford University\\
Stanford, CA 94305\\
\printead{e1}}

\end{aug}

\begin{abstract}
We define the spectral gap of a Markov chain on a finite state space as the second-smallest singular value of the generator of the chain, generalizing the usual definition of spectral gap for reversible chains. We then define the relaxation time of the chain as the inverse of this spectral gap, and show that this relaxation time can be characterized, for any Markov chain, as the time required for convergence of empirical averages. This relaxation time is related to the Cheeger constant and the mixing time of the chain through inequalities that are similar to the reversible case, and the path argument can be used to get upper bounds. Several examples are worked out. An interesting finding from the examples is that the time for convergence of empirical averages in nonreversible chains can often be substantially smaller than the mixing time.
\end{abstract}



\end{frontmatter}

\section{Introduction}
\subsection{Setting}
Let $X_0,X_1,\ldots$ be a time-homogeneous stationary Markov chain on a finite state space $\ms$. Let $P$ be the transition matrix of the chain and $L = I-P$ be the generator. Let $\mu$ be an invariant probability measure of our chain, such that $\mu(x)>0$ for all $x\in \ms$. The measure $\mu$ defines an inner product on $\cc^\ms$ in the usual way: For two functions $f,g:\ms\to \C$, 
\begin{align}\label{ipdef}
\smallavg{f,g} := \sum_{x\in \ms} f(x)\overline{g(x)}\mu(x).  
\end{align}
We will denote the norm induced by this inner product by $\|\cdot\|$. 
\begin{defn}
We define the spectral gap of the chain, $\gamma$, to be the second-smallest singular value of $L$ with respect to the above inner product on $\C^{\ms}$ (counting singular values with multiplicities), and call $\tau := 1/\gamma$ the relaxation time of the chain. 
\end{defn}
If the chain is reversible, then $P$ has eigenvalues $1=\lambda_1\ge \lambda_2\ge \cdots \ge \lambda_{|\ms|}\ge -1$ (repeated by multiplicities). Recall that in this case, the spectral gap is defined to be $1-\lambda_2$, which is the same as our definition of spectral gap. The relaxation time $\tau_{\textup{rel}}$ for a reversible chain is defined to be the inverse of the absolute spectral gap $1- \max\{\lambda_2, |\lambda_{|\ms|}|\}$. 
Under the condition $\lambda_2 \ge |\lambda_{|\ms|}|$, which can be easily arranged by making the chain sufficiently lazy, the absolute spectral gap is equal to the spectral gap. In this situation, the classical relaxation time $\tau_{\textup{rel}}$ equals our relaxation time $\tau$. 

Early readers of this paper had trouble understanding the difference between the present notion of spectral gap and the singular value analysis already in use for nonreversible chains. The difference is this: the present paper uses the singular values of $I-P$ while earlier analyses use the singular values of $P$. For nonreversible chains, there is no obvious relation between these. Of course, they have a simple relation for reversible chains. The development below shows how these new singular values can be a useful adjunct to Markov chain analysis.


\subsection{Main result}
Let $\mu g$ denote the $\mu$-average of a function $g:\ms \to \R$, that is, 
\[
\mu g := \sum_{x\in \ms} g(x)\mu(x). 
\]
Given any $n\ge1$, let $\mu_n g$ denote the empirical average
\[
\mu_n g := \frac{1}{n} \sum_{i=0}^{n-1} g(X_i).
\]
For a random variable $Z$, let $\|Z\|_{L^2} := (\E|Z|^2)^{1/2}$ denote its $L^2$ norm (not to be confused with the norm on $\ms$ induced by $\mu$). Taking our Markov chain $X_0,X_1,\ldots$ to be stationary (i.e., $X_0\sim \mu$), we define
\[
\Delta_n := \sup_{g\,:\, \|g-\mu g\|=1} \|\mu_n g - \mu g\|_{L^2}. 
\]
In other words, $\Delta_n$ measures how concentrated the random variable $\mu_n g$ is around the constant $\mu g$, if the starting state is drawn from $\mu$ and $g$ is taken to be the worst possible function. 
The following theorem, which is the  main result of this article, shows that $\Delta_n$ is small if and only if $n\gg \tau$. This gives a characterization of our relaxation time as the time required for convergence of empirical averages to their limiting values in $L^2$.
\begin{thm}\label{avgthm}
For any  $n\ge 1$,
\begin{align}\label{avg1}
\Delta_{n} \le \sqrt{\frac{4\tau}{n}}.
\end{align}
Conversely, for any $n\le \tau/3$, 
\begin{align}\label{avg2}
\Delta_{n} \ge \frac{1}{132}.
\end{align}
Lastly, for any $n\ge 1$, we have the lower bound
\begin{align}\label{avg3}
\max_{n\le k\le 2n} \Delta_k \ge \frac{\tau}{2n+3\tau}. 
\end{align}
\end{thm}
Theorem \ref{avgthm} is proved in Section \ref{avgproof}. In the remainder of this section, we study the relationships of our spectral gap with various standard quantities, like the spectral gaps of reversibilized chains, the mixing time in total variation distance, and the Cheeger constant. We also develop a version of the path argument for getting lower bounds on our spectral gap. At the end of the section we discuss the relevant literature. Examples are worked out in Section \ref{examplesec} and all proofs are in Section \ref{proofsec}.

An interesting finding from the examples is that for nonreversible chains, our relaxation time can be significantly smaller than the mixing time of the chain. For a flavor of such results, let us preview one example whose details are presented in Section \ref{torisec}. Consider a random walk on the torus $(\Z/N\Z)^2$ which, at each turn, either moves one step up with probability $1-p$ or moves one step to the right with probability $p$, wrapping around upon reaching the boundary. For any $p\in (0,1)$, the mixing time of this walk is of order $N^2$. Our relaxation time $\tau$ is also of order $N^2$ if $p$ is rational. If $p$ is an irrational algebraic number of degree two, such as $1/\sqrt{2}$, then we show that $\tau$ is of order $N^{4/3}$. Moreover, we show that $N^{4/3}$ is the smallest possible order of $\tau$ for any choice of $p$. In particular, this shows that it is possible to design a random walk on $(\Z/N\Z)^2$ that takes only local steps but empirical averages converge in time $N^{4/3}$ instead of $N^2$.

\subsection{Relation to the reversible case}
For reversible chains, recall the standard result that
\begin{align*}
\tilde{\Delta}_n := \sup_{g\,:\, \|g-\mu g\|=1} \|P^n g - \mu g\| = \biggl(1-\frac{1}{\tau_{\textup{rel}}}\biggr)^n,
\end{align*}
where $\tau_{\textup{rel}}$ is the relaxation time defined earlier. Thus, for reversible chains, the relaxation time is the time at which $\tilde{\Delta}_n$ becomes small. This is similar to $\Delta_n$ becoming small at our relaxation time $\tau$, as shown in Theorem \ref{avgthm}. When $\lambda_2 \ge |\lambda_{|\ms|}|$, we have $\tau=\tau_{\textup{rel}}$, and therefore $\Delta_n$ and $\tilde{\Delta}_n$ become small at the same time. Thus, for reversible chains, the standard definition of relaxation time is the same as our definition, provided that we make the chain lazy enough so that the absolute spectral gap is equal to the spectral gap.

For reversible chains, the relaxation time can be characterized as the time required for decorrelation of $L^2$ functions. To be precise, if $X_0,X_1,\ldots$ is a reversible Markov chain with stationary distribution $\mu$ and relaxation time $\trel$, then for any real-valued $f,g\in L^2(\mu)$ and any $n$, 
\begin{align*}
\textup{Corr}(f(X_0), g(X_n)) &\le \biggl(1-\frac{1}{\tau_{\textup{rel}}}\biggr)^{n}, 
\end{align*}
and equality is attained if we take $f=g=$ the eigenvector of $P$ corresponding to the eigenvalue with the second largest magnitude (i.e., either $\lambda_2$ or $\lambda_{|\ms|}$). For nonreversible chains, this characterization of the relaxation time as the time required for decorrelation is no longer true. For example, consider the Markov chain on $\zz/N\zz$ which either stays where it is with probability $1/2$, or moves one step to the right with probability $1/2$. We will see in Section~\ref{examplesec} that the relaxation time (according to our definition) of this nonreversible Markov chain is of order $N$. But it is not hard to show that the time to decorrelate is of order $N^2$, because $X_n = X_0 + n/2 + O(\sqrt{n}) \bmod N$, which implies that $X_n$ carries information about $X_0$ all the way up to $n$ of order $N^2$.

\subsection{Relation to spectral gaps of reversibilized chains}
Let $P$ be the transition matrix of a Markov chain on a finite state space with invariant measure $\mu$. Let $P^*$ denote the adjoint of $P$ with respect to the inner product \eqref{ipdef} induced by $\mu$. The multiplicative reversibilization of is this chain is the reversible Markov chain with transition matrix $M = PP^*$, and the additive reversibilization is the chain with transition matrix $A = \frac{1}{2}(P+P^*)$. (It is not hard to show that $P^*$, $A$ and $M$ are Markov transition matrices, and that $\mu$ is an invariant measure for all three.) The following result connects the spectral gap of the reversibilized chains with the original chain. 
\begin{thm}\label{revthm}
Let $\gamma$ be the spectral gap of the Markov chain with transition matrix $P$ (according to our definition), and let $\gamma_M$ and $\gamma_A$ be the spectral gaps of the multiplicative and additive reversibilizations of $P$, defined above. Then $\frac{1}{2}\gamma_A\le \gamma \le \sqrt{2\gamma_A}$ and $\frac{1}{2}\gamma_M\le \gamma$. Moreover,  if $P(x,x)\ge 1/2$ for $x\in \ms$, then we also have $\gamma \le \sqrt{2\gamma_M}$.
\end{thm}
This result will be proved in Section \ref{revproof}. In the above, it is clear that the lower bounds cannot be improved other than perhaps the constants, because if our Markov chain is reversible, then $\gamma_A=\gamma$ and $\gamma_M = 1-(1-\gamma)^2\le 2\gamma$. Interestingly, the upper bounds also cannot be improved. To see this, consider the Markov chain on $\Z/N\Z$ that stays where it is with probability $1/2$ and moves one step to the right with probability $1/2$. It is not hard to show that $\gamma$ is of order $1/N$ for this chain, whereas $\gamma_M$ and $\gamma_A$ are both of order $1/N^2$. The condition that $P(x,x)=1/2$ for all $x$ (or at least some holding) is also necessary for the lower bound on $\gamma_M$. This can be seen by considering the Markov chain on $\Z/N\Z$ that always moves one step to the right. For this chain too, $\gamma$ is of order $1/N$. But the $M$-chain does not move at all, and therefore $\gamma_M = 0$. 

\subsection{Relation to mixing time}
Recall that the total variation distance between two probability measures $\mu$ and $\nu$ on the finite set $\ms$ is defined as
\[
\|\mu-\nu\|_{\textup{TV}} := \max_{A\subseteq \ms} |\mu(A)-\nu(A)|,
\]
and given $\ve >0$, the mixing time $\tau_{\textup{mix}}(\ve)$ of our Markov chain (with stationary distribution $\mu$) is defined as 
\[
\tau_{\textup{mix}}(\ve) := \inf\{n \ge 0: \|P^n_x - \mu\|_{\textup{TV}}\le \ve\},
\]
where $P^n_x$ is the law of $X_n$ given $X_0=x$. The mixing time, roughly speaking, is the number of steps required for the chain to become approximately independent of its starting state, in the sense that no test can reliably distinguish between one starting state and another. For reversible chains, the relaxation time and the mixing time satisfy the pair of inequalities
\begin{align}\label{relineq}
(\tau_{\textup{rel}} - 1)\log \frac{1}{2\ve} \le \tau_{\textup{mix}}(\ve) \le \tau_{\textup{rel}} \log \frac{1}{\ve \mu_{\textup{min}}},
\end{align}
where $\mu_{\textup{min}}:= \min_{x\in \ms} \mu(x)$ and $\ve \in (0,1/2)$. (For a proof, see \cite[Chapter 12]{levinetal09}.) These bounds are often suboptimal, and may be far bigger than $\tmix(\ve)$; this happens in card shuffling~\cite{diaconisshahshahani81}, where $\mu_{\textup{min}}$ is $1/N!$, $N$ being the number of cards. The following result gives analogous bounds for our relaxation time. The proof is in Section \ref{mixproof}. 
\begin{thm}\label{mixthm}
Take any $\ve \in (0,1/5)$. Let $\tau_{\textup{mix}}(\ve)$ be defined as above, and let $\tau$ be the relaxation time according to our definition. Then 
\begin{align*}
\tau \le \biggl(\frac{4}{\log(2/(1+4\ve+2\ve^2))} + 2\biggr)  \tmix(\ve). 
\end{align*}
Moreover, if $P(x,x)\ge 1/2$ for all $x$, then for any $\ve \in (0,1/2)$,
\[
\tmix(\ve) \le 1 + 12 \tau^2\log \frac{1}{2\ve \mumin}.
\]
\end{thm}
Note the important difference with \eqref{relineq} that the upper bound on $\tmix(\ve)$ in the above result is of order $\tau^2$ instead of $\tau$. This is actually necessary, as can be seen by considering the chain on $\Z/N\Z$ that moves to the right with probability $1/2$ and stays where it is with probability $1/2$. The mixing time of this chain is of order $N^2$, whereas our relaxation time is of order $N$. The condition that $P(x,x)\ge 1/2$ for all $x$ (or at least some holding) is also necessary, as can be seen by considering the chain on $\Z/N\Z$ that moves to the right at each step, which never mixes but has $\tau$ of order $N$.

In Section \ref{examplesec}, we will see nontrivial examples of nonreversible Markov chains where our relaxation time is significantly smaller than the mixing time. For reversible chains, there are many examples where this is already known to be true. The random transpositions walk on the permutation group $S_N$ is one of the first such examples, where Diaconis and Shahshahani~\cite{diaconisshahshahani81} showed that the relaxation time is of order $N$ and the mixing time is of order $N\log N$. Another early example is the simple random walk on a high dimensional hypercube~\cite{aldous83}. Numerous other examples are now available in the literature. In all such examples, empirical averages converge much before the chain mixes.

\subsection{Relation to the Cheeger constant}
As before, let $X_0,X_1,\ldots$ be a stationary Markov chain on a finite state space $\ms$ with transition matrix $P$, invariant measure $\mu$, and spectral gap (according to our definition) $\gamma$. For $x,y\in \ms$, define 
\[
Q(x,y) := \mu(x) P(x,y). 
\]
That is, $Q$ is the joint distribution of $(X_0,X_1)$. For $A, B\subseteq \ms$, let
\[
Q(A, B) := \P(X_0\in A, \, X_1\in B) = \sum_{x\in A, \, y\in B} Q(x,y). 
\]
Let $A^c$ denote the complement of $A$ in $\ms$, that is, the set $\ms \setminus A$. The Cheeger constant (or the bottleneck ratio) of the chain is defined as 
\begin{align}\label{cheegerdef}
\xi := \min_{A\subseteq \ms, \, 0< \mu(A)\le 1/2} \frac{Q(A,A^c)}{\mu(A)}. 
\end{align}
A famous property of the Cheeger constant is that for reversible Markov chains, it can be used to get upper and lower bounds on the spectral gap through the following inequalities~\cite{lawlersokal88, jerrumsinclair89}:
\begin{align}\label{cheegerineq}
\frac{\xi^2}{2}\le \gamma \le 2\xi.
\end{align}
(For a proof, see \cite[Theorem 13.14]{levinetal09}.) If the chain is made sufficiently lazy so that the spectral gap equals the absolute spectral gap, then the above inequalities give estimates for the relaxation time in terms of the Cheeger constant.

For nonreversible chains, the Cheeger constant can be used to give lower bounds on the mixing time. For example, \cite[Theorem 7.3]{levinetal09}  says that for any Markov chain,
\[
\tmix(1/4)\ge \frac{1}{4\xi}. 
\]
For random walks on directed graphs, an analogue of the Cheeger constant was formulated by Chung~\cite{chung05}, who used it to obtain upper bounds on the mixing time. However, for general nonreversible Markov chains, there seems to be no analogue of \eqref{cheegerineq} in the literature. The following theorem gives such a pair of inequalities for our spectral gap.
\begin{thm}\label{cheegerthm}
Let $X_0,X_1,\ldots$ be a Markov chain on a finite state space with invariant measure $\mu$ and spectral gap $\gamma$ (according to our definition). Let $\xi$ be the Cheeger constant of the chain,  defined in \eqref{cheegerdef}. Then 
\[
\frac{\xi^2}{16}\le \gamma \le 32\xi.
\]
\end{thm}
Theorem \ref{cheegerthm} is proved in Section \ref{cheegerproof}. Note that the bounds are exactly the same as in the reversible case \eqref{cheegerineq}, but with worse constants. It is not clear if the constants can be improved. 

\subsection{Path argument for nonreversible chains}\label{pathsec}
The following theorem shows that the path argument of Diaconis and Stroock~\cite{diaconisstroock91} can be used to get lower bounds on our spectral gap for nonreversible chains. As usual, we consider a stationary Markov chain $X_0,X_1,\ldots$ on a finite state space $\ms$, transition matrix $P$, and invariant probability measure $\mu$. 
\begin{thm}\label{paththm}
Let $E$ be the set of ordered pairs $(x,y)$ such that $Q(x,y) := \mu(x)P(x,y)$ is strictly positive. The set $E$ defines a directed graph on $\ms$. Suppose that for each pair of distinct points $x,y\in \ms$, we have a directed path $\Gamma_{x,y}$ from $x$ to $y$ in this graph.  Let
\[
B:= \max_{e\in E}\frac{1}{Q(e)}\sum_{x,y\in \ms: \,\Gamma_{x,y}\ni e} \mu(x)\mu(y) |\Gamma_{x,y}|
\]
Then our spectral gap $\gamma$ satisfies $\gamma \ge 1/B$.
\end{thm}
Theorem \ref{paththm} is proved in Section \ref{pathproof}. The proof is almost identical to the proof in the reversible case, with one small twist.

One caveat about Theorem \ref{paththm} is that it can give suboptimal results in examples where the relaxation time of $P$ according our definition is much smaller than the relaxation time of the additive reversibilization $A = \frac{1}{2}(P+P^*)$. For example, this happens for biased random walks on $(\Z/N\Z)^d$. Also, a comparable version of Theorem \ref{paththm} can be derived by combining the usual path argument for $A$~\cite[Corollary 13.24]{levinetal09} and the inequality $\gamma \ge \frac{1}{2}\gamma_A$ from our Theorem \ref{revthm}. 

\subsection{Relation to eigenvalues of the transition matrix}
In general, there is no simple relation between the eigenvalues of the transition matrix $P$ and the singular values of the generator $L$ when the Markov chain is nonreversible. But there is one important exception, when $P$ is a normal matrix with respect to the inner product \eqref{ipdef}, meaning that $P^*P = PP^*$. Under this condition, the spectral theorem for finite dimensional complex inner product spaces implies that $P$ can be decomposed as $UDU^*$ for some diagonal matrix $D$ and some matrix $U$ that is unitary with respect to the inner product \eqref{ipdef} on $\C^\ms$. Thus, $L = U(I-D)U^*$, and hence 
\begin{align}\label{lprel}
L^*L = U(I-D)^*(I-D)U^*.
\end{align}
Let $\lambda_0,\lambda_1,\ldots, \lambda_{|\ms|-1}$ be the diagonal elements of $D$, arranged such that 
\begin{align}\label{arrange}
0=|1-\lambda_0|\le |1-\lambda_1|\le \cdots \le |1-\lambda_{|\ms|-1}|.
\end{align}
Note that $\lambda_0, \lambda_1,\ldots$ are the eigenvalues of $P$, repeated by multiplicities. The equation \eqref{lprel} shows that $|1-\lambda_0|, |1-\lambda_1|,\ldots$ are the singular values of $L$. Thus, we arrive at the following theorem.
\begin{thm}\label{revnormthm}
Suppose that $P^*P = PP^*$. Let $\lambda_0, \lambda_1,\ldots$ be the eigenvalues of $P$, repeated by multiplicities, and arranged in such a way that \eqref{arrange} holds. Then $|1-\lambda_1|$ is the second-smallest singular value of $L$. 
\end{thm}
We will later see some examples where $P$ is not reversible but normal, and will use Theorem~\ref{revnormthm} to study the relaxation time. 

\subsection{Related works}
Various notions of spectral gap for nonreversible chains have been proposed over the years. One of the earliest definitions is due to Fill~\cite{fill91}, who considered the additive and multiplicative reversibilizations of $P$ with respect to the inner product \eqref{ipdef}, and showed that the spectral gaps of these self-adjoint operators can be used to get upper bounds on the mixing time of the original chain. For example, Fill shows that if $\beta$ is the second largest eigenvalue of $PP^*$, then for any $x$ and $n$,
\begin{align}\label{fillbd}
\|P_x^n - \mu\|_{\textup{TV}} \le \frac{\beta^{n/2}}{2\sqrt{\mu(x)}}.
\end{align}
A different approach, due to Kontoyiannis and Meyn~\cite{kontoyiannismeyn03, kontoyiannismeyn12}, defines spectral gap in terms of the spectrum of $P$. For nonreversible chains, the spectrum $S$ of $P$ is a subset of the unit disk in the complex plane. According to the definition in \cite{kontoyiannismeyn03}, the chain has a spectral gap if  there exists $\ve_0 > 0$ such that $S\cap \{z : |z| \ge  1 - \ve_0\}$ is finite, and contains only poles of finite multiplicity. It is shown in \cite{kontoyiannismeyn03, kontoyiannismeyn12}, with quantitative bounds, that the existence of such a spectral gap implies geometric ergodicity for the corresponding chains.

Yet another definition, due to Paulin~\cite{paulin15}, is that of the pseudo-spectral gap. The pseudo-spectral gap $\gamma_{\textup{ps}}$ is defined as the maximum, over all $k\ge 1$, of the spectral gap of the reversible chain $(P^*)^k P^k$ divided by $k$. It is shown in \cite{paulin15} that the pseudo-spectral gap can be used to get upper bounds on the mixing times of Markov chains, as well as (usually non-matching) lower bounds. For example, Paulin shows that for any $\ve\in (0,1/2)$,
\[
\frac{1-2\ve }{\gamma_{\textup{ps}}} \le \tau_{\textup{mix}}(\ve) \le \frac{1+2|\log(2\ve)| + |\log \mumin|}{\gamma_{\textup{ps}}}.
\]
For applications of the pseudo-spectral gap and further theoretical advances, see \cite{wolferkontorovich19}. Incidentally, it turns out that our spectral gap $\gamma$ is bounded below by  $\frac{1}{2}\gamma_{\textup{ps}}$. This is because, as we will see in the proof of Theorem \ref{mixthm} (specifically, equation \eqref{maintau}), if $\gamma_k$ is the spectral gap of $P^k$ (according to our definition), then $\gamma_k\le k\gamma$. But by Theorem \ref{revthm}, $\gamma_k\ge$ one-half times the spectral gap of $(P^k)^*P^k$. Dividing by $k$ and then taking supremum over $k$ shows that 
\[
\gamma\ge \frac{1}{2}\gamma_{\textup{ps}}.
\]
Techniques for using singular values of transition matrices to obtain bound on rates of convergence of inhomogeneous Markov chains were devised by Saloff-Coste and Z\'{u}\~{n}iga~\cite{saloffcostezuniga07}. The main quantity appearing in these bounds is the product of the second-largest singular values of the transition matrices, generalizing \eqref{fillbd} to inhomogeneous chains. Note that for nonreversible chains, these singular values have no obvious relation to the singular values of the generator that we are using in this paper.

A notion of spectral gap for random walks on directed graphs via the eigenvalues of a symmetrized generator was defined by Chung~\cite{chung05}, who then used this spectral gap to get upper bounds on the mixing times of such walks.

In all of the above, the appearance of $\mumin$ or $\mu(x)$ results in a suboptimal bound  on the mixing time (which is sometimes severely suboptimal). Diaconis and Saloff-Coste~\cite{diaconissaloffcoste96} developed methods for using the singular values of $P$ to get the correct rates of convergence in total variation distance in certain classes of nonreversible problems. 

It is an easy consequence of the decorrelation property that empirical averages converge at the relaxation time for reversible chains. Going beyond $L^2$ bounds, a Chernoff-type concentration inequality was proved by Gillman~\cite{gillman93, gillman98} for empirical averages of reversible chains which depends only on the spectral gap of the chain. This was later improved by Dinwoodie~\cite{dinwoodie95} and Lezaud~\cite{lezaud98}, with the latter paper having an extension to nonreversible chains via the multiplicative reversibilization. See also Joulin and Ollivier~\cite{joulinollivier10} for a different approach and further references. It would be interesting to see whether a similar concentration inequality can be proved in terms of the relaxation time defined in this paper.

There is also a large literature on mixing of nonreversible Markov chains which avoids the question of defining a spectral gap and directly analyzes the chains by other methods. For a sampling of such results in chronological order, see Hildebrand~\cite{hildebrand97}, Chen et al.~\cite{chenetal99}, Diaconis et al.~\cite{diaconisetal00}, Neal~\cite{neal04}, and Chen and Hwang~\cite{chenhwang13}. 

The path argument in the form given here was developed for reversible chains by Diaconis and Stroock~\cite{diaconisstroock91}, with variants developed in Jerrum and Sinclair~\cite{jerrumsinclair89}, Quastel~\cite{quastel92} and Diaconis and Saloff-Coste~\cite{diaconissaloffcoste93}. Different versions of the path argument are also available non-reversible chains~\cite[Section 4]{dyeretal06}. 

In very interesting recent work, Hermon~\cite{hermon23} has shown that the relaxation time of a reversible Markov chain on a finite state space is characterized, up to universal constant multiples, by the worst case expected hitting time of a large set starting from outside that set. Hermon also has a similar characterization of a certain notion of relaxation time of non-reversible chains, which appears to be equivalent to our notion of relaxation time. See~\cite{hermon23} for details. 

In the next section, we will see a number of examples where our relaxation time is small compared to the relaxation times of closely related reversible chains. There is a sizable body of work on speeding up reversible Markov chains by replacing them with related nonreversible chains. For results and pointers to this literature, see Neal~\cite{neal04} and Diaconis and Miclo~\cite{diaconismiclo13}. 

\section{Examples}\label{examplesec}
\subsection{Random walks on the discrete circle}
Consider a random walk on the discrete circle $\zz/N\zz$, which takes a jump of size $a_i$ with probability $p_i$, where $a_1,\ldots,a_k$ are distinct elements of  $\zz/N\zz$ and $p_1,\ldots,p_k$ are positive real  numbers summing to $1$. The relaxation time, according to our definition, can be exactly computed, as shown by the following result.
\begin{thm}\label{formthm}
For the above Markov chain, the relaxation time according to our definition is given by the formula
\[
\tau =\max_{1\le j\le N-1} \biggl|1- \sum_{r=1}^k p_r e^{2\pi \I ja_r/N}\biggr|^{-1}. 
\]
\end{thm}

The proof of Theorem \ref{formthm} is in Section \ref{formproof}. 
For a quick example, consider the case of the simple symmetric random walk, where $k=2$, $a_1=1$, $a_2 = -1 \bmod N$, and $p_1=p_2=1/2$. A simple calculation using the above formula yields
\begin{align*}
\tau &= \max_{1\le j\le N-1} \biggl|1-\frac{1}{2}(e^{2\pi \I j/N} + e^{-2\pi \I j/N})\biggr|^{-1}\\
&= \max_{1\le j\le N-1} \frac{1}{1-\cos(2\pi j/N)} = \frac{1}{1-\cos(2\pi /N)}.
\end{align*}
This shows that the relaxation time is of order $N^2$, which is a well-known result. On the other hand, if we consider the biased random walk where $k=2$, $a_1 =0$, $a_2=1$, and $p_1=p_2=1/2$, then 
\begin{align*}
\tau &= \max_{1\le j\le N-1} \biggl|1-\frac{1}{2}(1+e^{2\pi\I j/N})\biggr|^{-1}\\
&= \max_{1\le j\le N-1}\frac{1}{|\sin (\pi j/N)|} = \frac{1}{|\sin (\pi/N)|}, 
\end{align*}
which grows like a multiple of $N$. It is not hard to show that the mixing time of the biased random walk is of order $N^2$. Thus, for the biased walk, empirical averages converge much before the chain mixes. 

What about typical $a_1,\ldots,a_k$, for a given $k$?  The following result shows that if $N$ is prime, then for any given $p_1,\ldots,p_k$, the value of $\tau$ for typical $a_1,\ldots, a_k$ is at most of order $N^{2/(k+1)}$. 
\begin{thm}\label{randthm}
Let $N$ be prime. Given $k$ and $p_1,\ldots,p_k$, if $a_1,\ldots, a_k$ are drawn independently and uniformly at random from $\zz/N\zz$, then for any $L>0$, $\P(\tau > L N^{2/(k+1)}) \le CL^{-(k+1)/2}$, where $C$ is a positive constant that depends only on $k$ and $p_1,\ldots,p_k$.
\end{thm}
Theorem \ref{randthm} is proved in Section \ref{randproof} using the formula from Theorem \ref{formthm}. We remark here that the mixing time of the above random walk for typical $a_1,\ldots,a_k$ is of order $N^{2/(k-1)}$, which improves to $N^{2/k}$ if we make the walk lazy. (See the survey of Hildebrand~\cite{hildebrand05} on random random walks for a proof and various other results of this flavor. See also \cite[Theorem D]{hermonolesker-taylor23},  which  considers the relaxation time and the absolute relaxation time of the reversible random walk on the random Cayley graph of an Abelian group $G$, obtained by picking $k$ elements of $G$ independently and uniformly at random. For continuous time random random walks and cutoff phenomena, see \cite{hough17, hermonolesker-taylor21}.) Theorem \ref{randthm} shows that the relaxation time, according to our definition, is typically of order $N^{2/(k+1)}$, which is much smaller than the mixing time. Thus, again, empirical averages converge much before the chain mixes. 

\subsection{Local random walks on tori}\label{torisec}
Take any $d\ge 1$. Let $(p_i)_{-d\le i\le d}$ be a set of nonnegative real numbers that sum to $1$. Consider the random walk on the discrete torus $(\Z/N\Z)^d$ which takes independent steps as $X_{n}=X_{n-1}$ with probability $p_0$ and $X_{n}=X_{n-1} \pm e_i$ with probability $p_{\pm i}$ for $i=1,\ldots,d$, where $e_1,\ldots,e_d$ are the standard basis vectors of $\R^d$. We assume that $p_i+p_{-i}>0$ for each $1\le i\le d$, because otherwise the Markov chain is not irreducible. Under this condition, the uniform distribution $\mu$ on the torus is the unique invariant measure of the walk. It is not hard to show that the mixing time of the chain is of order $N^2$ for any choice of $p_i$'s satisfying the above constraints.

The relaxation time shows more interesting behavior than the mixing time. We have already considered the case $d=1$ in the previous example. In particular, it is easy to deduce from Theorem \ref{formthm} that when $d=1$, $\tau$ is of order $N^2$ whenever the walk is reversible and of order $N$ whenever it is nonreversible. The following theorem shows that the situation is more subtle when~$d\ge 2$. 
\begin{thm}\label{pathex}
Suppose that $p_i+p_{-i}>0$ for all $1\le i\le d$. Then the relaxation time $\tau$  (according to our definition) of the above random walk is bounded above by $C_1N^2$, where $C_1$ depends only on $d$ and the $p_i$'s. Moreover if $d\ge 2$, and $p_i=p_{-i}$ for some $i$ or $p_i-p_{-i}$ is rational for at least two $i$'s, then $\tau \ge C_2N^2$, where $C_2$ is a positive constant that depends only on $d$ and the $p_i$'s.
\end{thm}
The proof of this theorem is in Section \ref{pathexproof}. The upper bound is proved via the path argument for nonreversible chains stated in Section \ref{pathsec}. The lower bound is proved by exhibiting a function $f$ with $\mu f = 0$ and $\|(I-P)f\|\le CN^{-2}\|f\|$. Note that if $P$ is reversible, then $p_i=p_{-i}$ for all $i$, and hence $\tau$ is at least of order $N^2$ by the above result. 

Is it possible to decrease $\tau$ in $d\ge 2$ with irrational transition probabilities, and if so, by how much? The following theorem shows that $\tau$ is at least of order $N^{2d/(d+1)}$ for any choice of $p_i$'s, and this lower bound is approximately achieved for almost all choices of $p_i$'s. 
\begin{thm}\label{genpithm}
Suppose that $p_i+p_{-i}>0$ for all $1\le i\le d$. Then $\tau \ge C_1 N^{2d/(d+1)}$, where $C_1$ is a positive constant depends only on $d$ and the $p_i$'s. Moreover, for almost every $(p_i)_{-d\le i\le d}$ (with respect to Lebesgue measure on the simplex), we have $\tau = N^{2d/(d+1) + o(1)}$ as $N\to\infty$. 
\end{thm}
Theorem \ref{genpithm} is prove in Section \ref{genpiproof}. The proof involves some techniques from Diophantine approximation. The lower bound is obtained by exhibiting a function $f$ such that $\mu f=0$ and $\|(I-P)f\|\le CN^{-2d/(d+1)} \|f\|$, and the upper bound uses the spectral decomposition of $P$ and a probabilistic argument via the Borel--Cantelli lemma.

Theorem \ref{genpithm} says that for almost all choices of $p_i$'s, the relaxation time is roughly of the same order as its best possible value. One might wonder if a specific, explicit value of $(p_i)_{-d\le i\le d}$ can be shown to have this property. By Theorem \ref{pathex}, rational values will not do. The following result gives an explicit choice when $d=2$.
\begin{thm}\label{explicitthm}
Let $\alpha\in (0,1)$ be an irrational algebraic number of degree two, such as $1/\sqrt{2}$. Consider the random walk on the two-dimensional torus $(\Z/N\Z)^2$ which moves either one step up with probability $1-\alpha$, or one step to the right with probability $\alpha$. The relaxation time of this walk (according to our definition) satisfies $C_1N^{4/3}\le \tau \le C_2 N^{4/3}$, where $C_1$ and $C_2$ are positive constants that depend only on $\alpha$.
\end{thm}
Theorem \ref{explicitthm} is proved in Section \ref{explicitproof}. The proof involves some elementary techniques from Diophantine approximation. For an interesting result that has some similarity with Theorem \ref{explicitthm}, in that the transition probability being irrational helps speed up the chain, see \cite[Theorem 1.5]{boczkowskietal18}. 

What if $\alpha$ is an irrational algebraic number of degree greater than two? It turns out that the relaxation time is still nearly optimal, as shown by the following theorem.
\begin{thm}\label{explicitthm2}
Let $\alpha\in (0,1)$ be any irrational algebraic number. Consider the random walk on the two-dimensional torus $(\Z/N\Z)^2$ which moves either one step up with probability $1-\alpha$, or one step to the right with probability $\alpha$. Given any $\ve > 0$, the relaxation time of this walk (according to our definition) satisfies $C_1N^{4/3}\le \tau \le C_2(\ve) N^{4/3+\ve}$, where $C_1$ is a  positive constant that depends only on $\alpha$, and $C_2$ is a positive constant that depends only on $\alpha$ and $\ve$.
\end{thm}
The proof of Theorem \ref{explicitthm2} --- presented in Section \ref{explicit2proof} --- uses Roth's theorem about rational approximations of algebraic numbers~\cite{roth55}. I thank Noga Alon for pointing out to me that combining Roth's theorem with the proof technique for Theorem \ref{explicitthm} yields Theorem~\ref{explicitthm2}. 


\subsection{The Chung--Diaconis--Graham chain}
The Chung--Diaconis--Graham chain~\cite{chungetal87} is a famous nonreversible Markov chain on $\Z/N\Z$, which proceeds as:
\[
X_n = 2X_{n-1} + \ve_n \bmod N,
\]
where $\ve_n$ are i.i.d.~random variables that are uniformly distributed in $\{-1,0,1\}$. It was shown in \cite{chungetal87} that the mixing time of this chain is at most of order $\log N \log \log N$. Moreover, it was shown in \cite{chungetal87} that for almost all odd $N$, the mixing time is of order $\log N$, and for an exceptional set of $N$, it is of order $\log N \log \log N$. It has been recently proved by Eberhad and Varj\'{u}~\cite{eberhardvarju21} that for almost all odd $N$, this chain exhibits the cutoff phenomenon in total variation distance, and the exact multiple of $\log N$ where the cutoff happens was also identified in the same paper. Here, we show that the relaxation time of this chain, according to our definition, is of order $\log N$ if $N$ is a prime. Thus, empirical averages are close to their limiting values at time $n$ if and only if $n\gg \log N$.
\begin{thm}\label{cdgthm}
There are positive constants $C_1$ and $C_2$ such that for any prime $N\ge 3$, the relaxation time of the Chung--Diaconis--Graham chain is bounded between $C_1\log N$ and $C_2\log N$. 
\end{thm}
Theorem \ref{cdgthm} is proved in Section \ref{cdgproof}. The upper bound is proved by showing that for any $f$ with $\mu f = 0$ and any $n$, $\|P^n f\|\le e^{-Cn/\log N}\|f\|$, and then using this to show that the second-smallest singular value of $I-P$ is at least of order $(\log N)^{-1}$. The lower bound is proved by explicitly constructing a function $f$ such that $\|f\|=1$ and $\|(I-P)f\|=O((\log N)^{-1})$. 


\subsection{A card shuffling scheme}
The following nonreversible card shuffling scheme was introduced by Diaconis and Saloff-Coste~\cite{diaconissaloffcoste96}. Take a deck of $N$ cards. With probability $1/3$, do nothing; with probability $1/3$, swap the top two cards; and with probability $1/3$, take out the bottom card and place it on top. The following result shows that for this walk on the symmetric group $S_N$, order $N^3$ steps are necessary and sufficient for convergence of empirical averages. 
\begin{thm}\label{cardthm}
Let $\tau$ be the relaxation time of the Markov chain on $S_N$ defined above. Then there are positive constants $C_1$ and $C_2$, independent of $N$, such that $C_1N^3 \le \tau \le C_2 N^3$. 
\end{thm}
Theorem \ref{cardthm} is proved in Section \ref{cardproof}. The upper bound is proved using a result from \cite{diaconissaloffcoste96}, which says that the second-largest singular value of the transition matrix is bounded above by $1-1/(41N^3)$. This result, together with our Theorem \ref{revthm}, gives the upper bound. The lower bound takes additional work. 


\subsection{Nonreversible random walks on groups}
Let $G$ be a finite group and $A$ be a generating subset of $G$ (meaning that any element of $G$ can be written as a product of elements of $A$). The random walk on $G$ defined by $A$ is the walk that proceeds by left-multiplying with a uniformly chosen element of $A$. The unique invariant measure of any such walk is the uniform distribution on $G$. If the set $A$ is closed under inversion, then the walk is a reversible Markov chain. There is a large literature on reversible walks on groups. The following result shows how the relaxation time of a natural reversibilization of a nonreversible walk can be used to get an upper bound on the relaxation time of the nonreversible walk.
\begin{thm}\label{groupthm}
Let $G$ be a finite group and $A$ be a generating subset of $G$. The relaxation time (according to our definition) of the random walk defined by $A$ is bounded above by twice the relaxation time of the reversible walk defined by $A\cup A^{-1}$. 
\end{thm}
This result is a straightforward consequence of Theorem \ref{revthm}, since the walk for the generating set $A\cup A^{-1}$ is the additive reversibilization of the original walk.

\section{Proofs}\label{proofsec}
\subsection{Proof of Theorem \ref{avgthm}}\label{avgproof}
First, suppose that $\tau=\infty$. Then the upper bound \eqref{avg1} is trivial. For the lower bound, note that $\tau=\infty$ means that $\gamma=0$, which implies that there is some $g:\ms \to \rr$ such that $\mu g=0$, $\|g\|=1$, and $Lg = 0$. The condition $Lg = 0$ is the same as $g = Pg$, which implies that $g = P^n g$ for all $n$. Thus,
\begin{align*}
\Delta_n^2 &\ge \E \biggl[\biggl(\frac{1}{n}\sum_{i=0}^{n-1} g(X_i)\biggr)^2\biggr]\\
&= \frac{1}{n^2}\sum_{i=0}^{n-1} \ee(g(X_i)^2) + \frac{2}{n^2}\sum_{1\le i<j\le n} \ee(g(X_i)g(X_j))\\
&= \frac{1}{n} + \frac{2}{n^2}\sum_{1\le i<j\le n} \ee(g(X_j) P^{j-i} g(X_j))\\
&= \frac{1}{n} + \frac{2}{n^2}\sum_{1\le i<j\le n} \ee(g(X_j)^2) = \frac{1}{n} + \frac{2}{n^2}{n\choose 2} = 1. 
\end{align*}
This proves the lower bounds \eqref{avg2} and \eqref{avg3} when $\tau=\infty$.

Next, suppose that $\tau<\infty$. Then $0$ is a singular value of $L$ with multiplicity $1$. This implies that $\mathrm{range}(L)$ has dimension $\ge |\ms|-1$.  But $\mathrm{range}(L)$ is contained in the $|\ms|-1$ dimensional subspace orthogonal to $\mu$. Thus, $\mathrm{range}(L)$ must be equal to this subspace. In particular, given any $g:\ms \to \rr$, there is a solution $f$ of the Poisson equation
\[
Lf = g - \mu g.
\]
Since $g - \mu g$ is orthogonal to $\mu$, the definition of $\tau$ implies that 
\begin{align}\label{fl2}
\|f\|\le \tau \|g-\mu g\|.
\end{align}
Let $f$ and $g$ be as above throughout the rest of this section. We will use $\E_x$ to denote expectation conditional on $X_0=x$. 
\begin{lmm}\label{upperlmm}
For any $n$,
\begin{align*}
\sum_x \mu(x)(\ee_x(\mu_n g - \mu g))^2 \le \frac{4\tau^2\|g-\mu g\|^2}{n^2}.
\end{align*}
\end{lmm}
\begin{proof}
Note that for any $x$,
\begin{align*}
\ee_x(g(X_n)-\mu g) &= \ee_x(f(X_n) - P f(X_n) ) = \ee_x(f(X_{n})-f(X_{n+1})).
\end{align*}
Thus,
\begin{align}
\ee_x(\mu_n g - \mu g)  &= \frac{1}{n}\sum_{k=0}^{n-1} \ee_x(g(X_k) - \mu g) \notag\\
&= \frac{1}{n}\sum_{k=0}^{n-1} \ee_x(f(X_{k})-f(X_{k+1}))\notag \\
&= \frac{1}{n}(f(x)-\ee_x f(X_n)).\label{maineq}
\end{align}
This shows that
\begin{align*}
\sum_x\mu(x)(\ee_x(\mu_ng - \mu g))^2 &= \frac{1}{n^2}\sum_x\mu(x)(\ee_x f(X_n) - f(x))^2\\
&\le  \frac{1}{n^2}\sum_x\mu(x)\ee_x[(f(X_n) - f(x))^2]\\
&\le \frac{1}{n^2}\sum_x \mu(x)(2\ee_x[f(X_n)^2] + 2f(x)^2)\\
&= \frac{4\|f\|^2}{n^2}.
\end{align*}
The proof is completed upon applying \eqref{fl2}.
\end{proof}
\begin{proof}[Proof of \eqref{avg1}]
Take any $g:\ms \to \rr$ such that $\mu g = 0$ and $\|g\|=1$. Then 
\begin{align*}
\sum_x \mu(x)\ee_x\biggl[\biggl(\frac{1}{n}\sum_{i=0}^{n-1} g(X_i)\bigg)^2\biggr]&= \frac{1}{n^2}\sum_x \sum_{i,j=0}^{n-1}  \mu(x) \ee_x[g(X_i)g(X_j)]\\
&= - \frac{1}{n^2}\sum_x\sum_{i=0}^{n-1} \mu(x) \ee_x[g(X_i)^2]\\
&\qquad + \frac{2}{n^2}\sum_x \sum_{i=0}^{n-1} \mu(x) \ee_x \biggl[g(X_i)\sum_{j=i}^{n-1} g(X_j)\biggr].
\end{align*}
The first term above is simply $-\|g\|^2/n$. For the second term, take any $0\le i\le n-1$. Define
\begin{align}
h(x) &:= \ee\biggl[\sum_{j=i}^{n-1} g(X_j)\biggl| X_i = x\biggr]\notag \\
&= \ee\biggl[\sum_{j=0}^{n-i-1} g(X_j)\biggl| X_0 = x\biggr]\notag \\
&= (n-i)\ee_x(\mu_{n-i} g).\label{hform}
\end{align}
Thus, 
\begin{align*}
\ee_x \biggl[g(X_i)\sum_{j=i}^{n-1} g(X_j)\biggr] &= \ee_x[g(X_i) h(X_i)]\\
&\le \sqrt{\ee_x[g(X_i)^2] \ee[h(X_i)^2]}.
\end{align*}
So, with an application of the Cauchy--Schwarz inequality, we get
\begin{align*}
&\sum_x  \mu(x) \ee_x \biggl[g(X_i)\sum_{j=i}^{n-1} g(X_j)\biggr] = \sum_x \mu(x)\ee_x[g(X_i)h(X_i)]\\
&\le \sqrt{\sum_x \mu(x) \ee_x[g(X_i)^2]\sum_x \mu(x) \ee_x[h(X_i)^2]} = \|g\| \|h\|.
\end{align*}
But by equation \eqref{hform} and Lemma \ref{upperlmm}, 
\begin{align*}
\|h\|^2 &= (n-i)^2 \sum_x \mu(x)( \ee_x(\mu_{n-i} g))^2 \le 4\tau^2\|g\|^2.
\end{align*}
Putting it all together, we get 
\begin{align*}
\sum_x \mu(x)\ee_x\biggl[\biggl(\frac{1}{n}\sum_{i=0}^{n-1} g(X_i)\bigg)^2\biggr] &\le \frac{4\tau\|g\|^2}{n}  - \frac{\|g\|^2}{n}.
\end{align*}
Taking supremum over $g$ completes the proof of \eqref{avg1}.
\end{proof}
It remains to prove \eqref{avg2} and \eqref{avg3}. To do this, we need the following lemma.
\begin{lmm}\label{deltalmm}
If $n= n_1+\cdots+n_k$, then
\[
n \Delta_n \le n_1\Delta_{n_1}+\cdots + n_k\Delta_{n_k}.
\]
Moreover, $\Delta_n \le 1$ for all $n$.
\end{lmm}
\begin{proof}
Take any $g:\ms \to \rr$ such that $\|g-\mu g\|=1$. Without loss of generality, suppose that $\mu g=0$.  Then by the triangle inequality for the $L^2$ norm,
\begin{align*}
n\|\mu_n g - \mu g\| &= \biggl\|\sum_{i=0}^{n-1} g(X_i)\biggr\|_{L^2}\\
&\le \biggl\|\sum_{i=0}^{n_1-1}g(X_i)\biggr\|_{L^2} + \biggl\|\sum_{i=n_1}^{n_1+n_2-1}g(X_i)\biggr\|_{L^2}+  \cdots + \biggl\|\sum_{i=n_1+\cdots+ n_{k-1}}^{n-1} g(X_i)\biggr\|_{L^2}\\
&\le n_1 \Delta_{n_1}+\cdots +n_k\Delta_{n_k},
\end{align*}
where the last inequality follows from the stationarity of the chain. Taking supremum over $g$ completes the proof of the first inequality in the statement of the lemma. The second inequality follows from the first by writing $n=1+1+\cdots+1$, and observing that $\Delta_1=1$. 
\end{proof}

\begin{proof}[Proofs of \eqref{avg2} and \eqref{avg3}]
Take any $f$ that minimizes $\|Lf\|$ subject to $\mu f = 0$ and $\|f\|=1$. Note that $f$ is real-valued, since the entries of $L$ are real. Also, note that $\gamma = \|Lf\|$. Let $g := Lf$. For each $n$, define the functions 
\begin{align*}
u_n(x) &:= f(x) - \frac{1}{n}\sum_{k=n}^{2n-1}\ee_x (f(X_k)), \\  
v_n(x) &:= \frac{1}{n}\sum_{k=n}^{2n-1}\ee_x (f(X_k)), \\ 
w_n(x) &:= \frac{1}{n}\sum_{k=0}^{n-1}\ee_x (g(X_k)). 
\end{align*}
Since $f = u_n+v_n$ and $\|f\|=1$, we have 
\begin{align}
1 &= \|f\|\le \|u_n\| + \|v_n\|.  \label{fxbd}
\end{align}
Next, note that $\mu g = 0$, and \eqref{maineq} holds for this $g$ and $f$. Thus, for any $x$ and $n$,
\begin{align}
u_n(x) &=  \frac{1}{n} \sum_{k=n}^{2n-1} (f(x) - \ee_x(f(X_k)))= \frac{1}{n}\sum_{k=n}^{2n-1} k w_k(x).\notag
\end{align}
Since $\|w_k\|\le\|g\| \Delta_k = \gamma \Delta_k$ for each $k$, this shows that 
\begin{align}\label{fxbd1}
\|u_n\| \le \frac{1}{n}\sum_{k=n}^{2n-1} k \|w_k\| \le 2n\gamma\max_{n\le k\le 2n-1} \Delta_k. 
\end{align}
On the other hand, the definitions of $v_n$ and $\Delta_n$, together with the fact that $\|f-\mu f\|=1$, show that
\begin{align}\label{fxbd2}
\|v_n\| &= \frac{1}{n}\|2n \mu_{2n} f - n \mu_n f\|\notag \\
&\le 2 \|\mu_{2n} f\| + \|\mu_n f\|\le 2\Delta_{2n} +\Delta_n.
\end{align}
Using \eqref{fxbd1} and \eqref{fxbd2} in \eqref{fxbd}, we get
\begin{align*}
1 &\le (2n\gamma + 3)\max_{n\le k\le 2n} \Delta_k.
\end{align*}
This proves \eqref{avg3}. Finally, take any $m\le \tau/3$. Let $n := 12m$. By the above inequality, we know that there is some $k$ between $n$ and $2n$ such that 
\[
\Delta_k \ge \frac{1}{2n\gamma + 3} = \frac{1}{24m\gamma +3} \ge \frac{1}{8\tau \gamma+3} = \frac{1}{11}. 
\]
Write $ k = qm+r$, where $q$ and $r$ are  the quotient and the remainder when dividing $k$ by $m$. Then by Lemma \ref{deltalmm},
\begin{align*}
 \Delta_k &\le \frac{qm}{k}\Delta_m + \frac{r}{k}\Delta_r\le \Delta_m + \frac{r}{k}\le \Delta_m + \frac{1}{12}.
\end{align*}
Combining the last two displays, we get $\Delta_m \ge 1/132$, proving \eqref{avg2}. 
\end{proof}

\subsection{Proof of Theorem \ref{revthm}}\label{revproof}
First, let us prove the lower bounds. Let $\lambda_M$ be the second-largest eigenvalue of $M$, so that $\gamma_M = 1-\lambda_M$ and the second-largest singular value of $P$ is $\sqrt{\lambda_M}$. Thus, for any $f:\ms \to \C$ such that $\mu f = 0$, we have 
\[
\|(I-P)f\| \ge \|f\|-\|Pf\|\ge (1- \sqrt{\lambda_M}) \|f\|.
\]
This shows that
\begin{align*}
\gamma &\ge 1-\sqrt{\lambda_M} = 1-\sqrt{1-\gamma_M}\ge 1-\biggl(1-\frac{1}{2}\gamma_M\biggr) = \frac{\gamma_M}{2},
\end{align*}
where we used the inequality $\sqrt{1-x}\le 1-\frac{1}{2}x$ that holds for all $x\in (0,1)$.  Next, define
\[
R := \frac{1}{2}(I+P),
\]
so that $R$ is a transition matrix, and the diagonal elements of $R$ are all $\ge 1/2$. Let $\beta$ be the second-largest singular value of $R$ and $\lambda$ be the second-largest eigenvalue of $\frac{1}{2}(R+R^*)$. Then by \cite[Corollary 2.9]{fill91}, $\beta^2 \le \lambda$. Thus, for any $f:\ms \to \C$ with $\mu f = 0$, we have
\[
\|(I-R)f\|\ge \|f\|-\|Rf\| \ge (1-\beta)\|f\|\ge (1-\sqrt{\lambda})\|f\|. 
\]
But, note that 
\[
I-R = \frac{1}{2}(I-P),
\]
and 
\[
\frac{1}{2}(R+R^*) = \frac{1}{2}(I+P_A). 
\]
Thus,  $\lambda = 1-\gamma_A/2$, and therefore,
\begin{align*}
\|(I-P)f\| &= 2\|(I-R)f\| \\
&\ge 2(\|f\|-\|Rf\|) \ge 2(1-\beta)\|f\|\\
&\ge 2(1-\sqrt{\lambda}) \|f\| = 2(1-\sqrt{1-\gamma_A/2})\|f\|\\
&\ge 2(1-(1-\gamma_A/4))\|f\| = \frac{\gamma_A\|f\|}{2}. 
\end{align*}
This shows that $\gamma\ge \frac{1}{2}\gamma_A$.

Next, let us prove the upper bounds on $\gamma$ in terms of $\gamma_A$ and $\gamma_M$. The following proof is due to Piyush Srivastava, and was communicated to me by Hariharan Narayanan. It improves on the proof in the first draft of this paper. Let $L := I-P$, $L_A:= I-A$ and $L_M := I-M$. Then note that 
\begin{align*}
LL^* &= I - P - P^* + PP^* \\
&= 2I - P - P^* - (I-PP^*) = 2L_A - L_M.
\end{align*}
Thus, for any $f$ with $\mu f =0$ and $\|f\|=1$, we have 
\begin{align*}
\smallavg{f, LL^*f} &= 2\smallavg{f, L_A f} - \smallavg{f, L_M f}.
\end{align*}
Since $M$ is the transition matrix of a Markov chain that is reversible with respect to the stationary measure $\mu$, $L_M$ is positive semidefinite with respect to the inner product $\smallavg{\cdot,\cdot}$. Thus, $\smallavg{f, L_M f}\ge 0$. This shows that
\[
\smallavg{f, LL^*f}\le  2\smallavg{f, L_A f}.
\]
Taking infimum over $f$ on both sides shows that $\gamma^2 \le 2\gamma_A$.  Next, if $P(x,x)\ge 1/2$ for all $x\in \ms$, then by \cite[Equation (2.17)]{fill91}, the second-largest eigenvalue of $M$ is bounded above by second-largest eigenvalue of $A$, or equivalently, $\gamma_M \ge \gamma_A$. Thus, $\gamma_M \ge  \frac{1}{2}\gamma^2$.

\subsection{Proof of Theorem~\ref{mixthm}}\label{mixproof}
We need the following lemma.
\begin{lmm}\label{linftylmm}
Let $P$ and $Q$ be two Markov transition matrices on a finite state space $\ms$ such that $\mu$ is an invariant probability measure for both. Let $R:= PQ$. Then
\[
\max_{x\in \ms} \|R_x-\mu\|_{\textup{TV}} \le 2 \max_{x\in \ms} \|P_x-\mu\|_{\textup{TV}}\cdot \max_{x\in \ms} \|Q_x-\mu\|_{\textup{TV}}. 
\]
\end{lmm}
\begin{proof}
Let $\|\cdot\|_{L^\infty}$ denote the $L^\infty$ norm on $\ms$. For any square matrix $A$ whose rows and columns are indexed by elements of $\ms$, let
\[
\|A\|_{L^\infty \to L^\infty} := \max\{\|Af\|_{L^\infty}: \|f\|_{L^\infty}\le 1\}. 
\]
From this definition, it is easy to deduce that for any $A$ and $B$,
\begin{align}\label{linftyineq}
\|AB\|_{L^\infty \to L^\infty} &\le \|A\|_{L^\infty \to L^\infty}\|B\|_{L^\infty \to L^\infty}.
\end{align}
On the other hand, it is also easy to show that 
\begin{align}\label{linftyform}
\|A\|_{L^\infty \to L^\infty} &= \max_{x\in \ms} \sum_{y\in \ms}|A(x,y)|. 
\end{align}
Let $M$ be the matrix defined as $M(x,y)=\mu(y)$. Then $M^2=M$, and using the fact that $\mu$ is an invariant measure for both $P$ and $Q$, we get $MP=PM=M$ and $MQ=QM=M$. Thus,
\[
R - M = PQ-M= (P-M)(Q-M).
\]
Therefore, by \eqref{linftyform} and \eqref{linftyineq},
\begin{align*}
\max_{x\in \ms} \|R_x-\mu\|_{\textup{TV}} &= \frac{1}{2}\max_{x\in \ms} \sum_{y\in \ms}|R(x,y) - \mu(y)| \\
&= \frac{1}{2}\|R-M\|_{L^\infty\to L^\infty} \\
&=  \frac{1}{2}\|(P-M)(Q-M)\|_{L^\infty\to L^\infty}\\
&\le  \frac{1}{2}\|P-M\|_{L^\infty\to L^\infty} \|Q-M\|_{L^\infty\to L^\infty}.
\end{align*}
The proof is now completed by another application of \eqref{linftyform}. 
\end{proof}
We are now ready to complete the proof of Theorem \ref{mixthm}.
\begin{proof}[Proof of Theorem \ref{mixthm}]
Take any $f:\ms \to \C$ such that $\mu f = 0$. Fix some $\ve\in (0,1/5)$, and let $n := \tmix(\ve)$. Let $\tau'$ be the relaxation time of $P^n$, so that
\begin{align}\label{firstineq}
\|f - P^n f\| \ge \frac{\|f\|}{\tau'}.
\end{align}
Now, for each $k$,
\[
\|P^k f - P^{k+1} f\|= \|P^k(I-P)f\|\le \|(I-P)f\|,
\]
since $P$ is an $L^2(\mu)$-contraction. Thus, 
\[
\|f-P^n f\| \le \sum_{k=0}^{n-1}\|P^k f - P^{k+1} f\|\le n\|f - P f\|.
\]
Combining this with \eqref{firstineq}, we get
\[
\|f-Pf\| \ge \frac{\|f\|}{n\tau'}.
\]
Thus, we have
\begin{align}\label{maintau}
\tau \le n\tau'.
\end{align}
Next, let $A$ be the additive reversibilization of $P^n$, that is,
\[
A := \frac{1}{2}((P^*)^n + P^n). 
\]
Then note that for any $x\in \ms$,
\begin{align*}
\|A^2_x - \mu \|_{\textup{TV}} &\le \frac{1}{4}(\|((P^*)^{2n})_x - \mu \|_{\textup{TV}} + \|((P^*)^nP^n)_x - \mu\|_{\textup{TV}}\\
&\qquad \qquad \qquad  +  \|(P^n(P^*)^n)_x - \mu\|_{\textup{TV}}  + \|P_x^{2n} - \mu\|_{\textup{TV}}).
\end{align*}
By Lemma \ref{linftylmm} and the fact that $n= \tmix(\ve)$, this gives us
\begin{align*}
\max_{x\in \ms} \|A^2_x - \mu \|_{\textup{TV}}  &\le \frac{1}{4}(1+ 4\ve + 2\ve^2) =: \ve' <\frac{1}{2},
\end{align*}
where the last inequality holds because $\ve < 1/5$. Let $\sigma_{\textup{rel}}$ and $\sigma_{\textup{mix}}(\cdot)$ denote the relaxation time and mixing time of $A$. The above inequality shows that 
\[
\sigma_{\textup{mix}}(\ve')\le 2.
\]
Since $A$ is a reversible transition matrix, this shows that by \eqref{relineq}, 
\[
\sigma_{\textup{rel}} \le \frac{2}{\log\frac{1}{2\ve'}} + 1. 
\]
But by Theorem \ref{revthm}, $\tau' \le 2\sigma_{\textup{rel}}$. By \eqref{maintau}, this completes the proof of the upper bound on $\tau$ in terms of $\tmix(\ve)$. 

Next, note that by Theorem \ref{revthm} and the assumption that $P(x,x)\ge 1/2$ for all $x$, we have $\gamma \le \sqrt{6\gamma_M}$, where $M= PP^*$. Now, $M$ is self-adjoint with respect to the inner product \eqref{ipdef}, and has the same eigenvalues as $P^*P$. Since $P^*$ and $P$ are both Markov transition matrices with invariant measure $\mu$, so is $P^*P$. Thus, for any $f$ such that $\mu f = 0$,
\begin{align*}
\|Pf\|^2 &= \smallavg{Pf, Pf} = \smallavg{f, P^*Pf} \le (1-\gamma_M)\|f\|^2 \le \biggl(1-\frac{\gamma^2}{6}\biggr)\|f\|^2.
\end{align*}
Iterating this, we get that for any $n$,
\begin{align}\label{pnf1}
\|P^n f\|\le  \biggl(1-\frac{\gamma^2}{6}\biggr)^{n/2}\|f\|.
\end{align}
Now, for any $f$, $x$ and $n$, 
\begin{align*}
|\E_x(f(X_n)) - \mu f| &= \frac{1}{\mu(x)} \mu(x)|\E_x(f(X_n)) - \mu f|\\
&\le  \frac{1}{\mu(x)}\sum_{y\in \ms} \mu(y)|\E_y(f(X_n)) - \mu f|\\
&\le \frac{1}{\mu(x)}\sqrt{\sum_{y\in \ms} \mu(y)|\E_y(f(X_n)) - \mu f|^2}\\
&= \frac{1}{\mu(x)}\|P_n f - \mu f\|. 
\end{align*}
Applying \eqref{pnf1} and taking supremum over $f:\ms \to [-1,1]$, we get
\[
\|P_x^n - \mu\|_{\textup{TV}} \le \frac{1}{2\mu(x)}  \biggl(1-\frac{\gamma^2}{6}\biggr)^{n/2}. 
\]
Take any $\ve\in (0,1)$. If $\tmix(\ve)=0$, there is nothing to prove. So, let us assume that $\tmix(\ve) \ge 1$. If $n= \tmix(\ve)-1$, then the left side above is $> \ve$ for some $x$. Note that $\mu(x)>0$, since $P(y,y)\ge 1/2$ for all $y$. Thus, the above inequality shows that 
\[
\biggl(1-\frac{\gamma^2}{6}\biggr)^{(\tmix(\ve)-1)/2} >  2\ve \mu(x)\ge 2\ve \mumin.
\]
Taking logarithms on both sides, we get
\[
\frac{\tmix(\ve) -1}{2} \le \frac{\log(2\ve \mumin)}{\log(1-\frac{\gamma^2}{6})}. 
\]
Since $\log (1-x)\le -x$ for all $x\in (0,1)$, this gives
\[
\frac{\tmix(\ve) -1}{2} \le \frac{6}{\gamma^2}\log \frac{1}{2\ve \mumin}.
\]
This completes the proof of the upper bound on $\tmix(\ve)$ in terms of $\tau$.
\end{proof}

\subsection{Proof of Theorem \ref{cheegerthm}}\label{cheegerproof}
Let $A:= \frac{1}{2}(P+P^*)$ be the additive reversibilization of $P$. Let $Q_A(x,y) := \mu(x)A(x,y)$. It is not hard to check that
\[
P^*(x,y) = \frac{\mu(y)}{\mu(x)} P(y,x). 
\]
Thus,
\begin{align*}
Q_A(x,y) &= \frac{1}{2}\mu(x) P(x,y) + \frac{1}{2}\mu(x)\frac{\mu(y)}{\mu(x)} P(y,x)\\
&= \frac{1}{2}(Q(x,y) + Q(y,x)),
\end{align*}
where $Q(x,y) := \mu(x)P(x,y)$. Thus, for any $B\subseteq \ms$ with $\mu(B) >0$, we have
\[
Q_A(B,B^c) = \frac{1}{2}(Q(B,B^c)+Q(B^c, B))\ge \frac{1}{2}Q(B,B^c). 
\]
This shows that $\xi_A \ge \frac{1}{2}\xi$, where $\xi_A$ is the Cheeger constant of the $A$-chain. But by Theorem~\ref{revthm}, $\gamma \ge \frac{1}{2}\gamma_A$, and by the inequality \eqref{cheegerineq} for the reversible chain $A$, $\gamma_A \ge \frac{1}{2}\xi_A^2$. Combining, we get
\[
\gamma \ge \frac{1}{2}\gamma_A \ge \frac{1}{4}\xi_A^2\ge \frac{1}{16}\xi^2.
\]
This proves the claimed lower bound on $\gamma$. Next, let us prove the upper bound. Let $A\subseteq \ms$ be a set where the minimum in the definition of $\xi$ is attained, so that $\mu(A)>0$.  For each $n\ge 0$, let
\begin{align*}
Y_n := 1_A(X_0)1_A(X_1)\cdots 1_A(X_n).
\end{align*}
Since $X_0,X_1,\ldots$ is a stationary Markov chain, 
\begin{align*}
\E|Y_n - Y_{n+1}| &= \P(X_0\in A, \ldots, X_n\in A, X_{n+1}\notin A)\\
&\le \P(X_n\in A, X_{n+1}\notin A) = Q(A, A^c)= \xi \mu(A).
\end{align*}
Thus, for any $n\ge 1$,
\begin{align*}
\E|Y_0-Y_{n-1}| &\le (n-1) \xi \mu(A).
\end{align*}
But
\[
\E|Y_0-Y_{n-1}| = \P(X_0\in A) - \P(X_0\in A, \ldots,X_{n-1}\in A). 
\]
Combining the last two displays, we get
\begin{align}\label{pnxi}
\P(X_0\in A, \ldots,X_{n-1}\in A) \ge (1-(n-1)\xi)\mu(A). 
\end{align}
Let 
\[
f(x) := \frac{1_A(x)-\mu(A)}{\sqrt{\mu(A)(1-\mu(A))}},
\]
so that $\mu f = 0$ and $\|f\|=1$. Now, if $X_0\in A,\ldots,X_{n-1}\in A$, then 
\[
\mu_n f = \frac{1}{n}\sum_{i=0}^{n-1}f(X_i) = \frac{1-\mu(A)}{\sqrt{\mu(A)(1-\mu(A)}} = \sqrt{\frac{1-\mu(A)}{\mu(A)}}. 
\]
Since $\mu(A)\le 1/2$, we can now use \eqref{pnxi} to conclude that
\begin{align*}
\|\mu_n f\|_{L^2}^2 &\ge \frac{1}{2\mu(A)} \P(X_0\in A,\ldots, X_{n-1}\in A)\ge \frac{1 - (n-1)\xi}{2}.
\end{align*}
On the other hand, by Theorem \ref{avgthm},
\[
\|\mu_n f \|_{L^2}^2 \le \frac{4}{\gamma n}. 
\]
Combining the last two displays, we get that for any $n\ge 1$ such that $1-(n-1)\xi \ge 0$,
\[
\gamma \le \frac{8}{(1 - (n-1)\xi)n}. 
\]
Choosing $n$ such that $n-1\le 1/2\xi \le n$, we get
\[
\gamma \le \frac{16}{n}\le 32\xi.
\]
This completes the proof of the upper bound.

\subsection{Proof of Theorem \ref{paththm}}\label{pathproof}
For $f:\ms \to \R$, define
\begin{align*}
\me(f) := \sum_{x, y\in \ms} Q(x,y)(f(x)-f(y))^2,
\end{align*}
where $Q(x,y) := \mu(x) P(x,y)$. Let $f_0$ be a function that minimizes $\me(f)$ under the constraints that $\|f\|=1$ and $\mu f=0$. Note that for any $f:\ms \to \R$, 
\begin{align*}
\me(f) &= \E[(f(X_0)-f(X_1))^2]\\
&= \E(f(X_0)^2) + \E(f(X_1)^2) - 2\E(f(X_0)f(X_1))\\
&= 2\E(f(X_0)^2) - 2\E(f(X_0)Pf(X_0))\\
&= 2\smallavg{f, (I-P)f} \le 2\|f\|\|(I-P)f\|=2\|(I-P)f\|. 
\end{align*}
The minimum value of the right side among all $f$ satisfying the above constraints is $2\gamma$. Thus, 
\begin{align}\label{meineq}
\me(f_0)\le 2\gamma.
\end{align}
For any directed edge $e= (a,b)\in E$, let $|\nabla f_0(e)| := |f_0(a)-f_0(b)|$. For $x=y$, let $\Gamma_{x,y}$ denote the empty set. Then note that by the Cauchy--Schwarz inequality, 
\begin{align*}
&\sum_{x, y\in \ms} \mu(x)\mu(y)(f_0(x)-f_0(y))^2 \le\sum_{x,y\in \ms} \mu(x)\mu(y) \biggl(\sum_{e\in \Gamma_{x,y}} |\nabla f_0(e)|\biggr)^2\\
&\le \sum_{x,y\in \ms} \mu(x)\mu(y) |\Gamma_{x,y}|\biggl(\sum_{e\in \Gamma_{x,y}} |\nabla f_0(e)|^2\biggr)\\
&= \sum_{e\in E} Q(e) |\nabla f_0(e)|^2\biggl(\frac{1}{Q(e)}\sum_{x,y\in \ms:\, \Gamma_{x,y}\ni e}\mu(x)\mu(y) |\Gamma_{x,y}|\biggr)\\
&\le \sum_{e\in E} Q(e) |\nabla f_0(e)|^2 \biggl(\max_{e'\in E}\frac{1}{Q(e')}\sum_{x,y\in \ms: \,\Gamma_{x,y}\ni e'} \mu(x)\mu(y) |\Gamma_{x,y}|\biggr) = B \sum_{e\in E} Q(e) |\nabla f_0(e)|^2.
\end{align*}
But note that
\begin{align*}
\sum_{e\in E} Q(e) |\nabla f_0(e)|^2 &= \sum_{a,b\in \ms} Q(a,b)(f_0(a)-f_0(b))^2 = \me(f_0). 
\end{align*}
Also, since $\mu f = 0$ and $\|f\|=1$, we have
\[
\sum_{x, y\in \ms} \mu(x)\mu(y)(f_0(x)-f_0(y))^2 = 2.
\]
Thus, by \eqref{meineq}, we get $\gamma \ge 1/B$.

\subsection{Proof of Theorem \ref{formthm}}\label{formproof}
Note that the transition matrix $P$ is a circulant matrix in this problem. It is easy to check that any circulant matrix is normal with respect to the ordinary inner product. Moreover, it is known that the eigenvalues of $P$, since it is a circulant matrix, are $\lambda_0,\ldots,\lambda_{N-1}$, where 
\[
\lambda_j = \sum_{r=1}^k p_r e^{2\pi \I ja_r/N}. 
\]
The result now follows by Theorem \ref{revnormthm}.

\subsection{Proof of Theorem \ref{randthm}}\label{randproof}
Fix $k$ and $p_1,\ldots, p_k$. Throughout this proof, $C, C_1,C_2,\ldots$ will denote any positive constants that depend only on $k$ and $p_1,\ldots,p_k$, whose values may change from line to line. Let $a_1,\ldots,a_k$ be drawn independently and uniformly at random from $\{1,\ldots, N-1\}$. 
For each $j$ and $r$, let $X_{j,r}$ be the unique element of $\{-(N-1)/2,\ldots,(N-1)/2\}$ such that $ja_r = X_{j,r} \bmod N$. Then  note that
\begin{align*}
\biggl|1- \sum_{r=1}^k p_r e^{2\pi \I ja_r/N}\biggr|^2 &= \biggl(\sum_{r=1}^k p_r (1-\cos(2\pi  ja_r/N))\biggr)^2 + \biggl(\sum_{r=1}^k p_r \sin(2\pi  ja_r/N)\biggr)^2\\
&= \biggl(\sum_{r=1}^k p_r (1-\cos(2\pi  X_{j,r}/N))\biggr)^2 + \biggl(\sum_{r=1}^k p_r \sin(2\pi  X_{j,r}/N)\biggr)^2. 
\end{align*}
Take any $\delta\in (0,1)$. Fix some $j$, and let $E_j$ be the event that the above quantity is less than $\delta^2$. If $E_j$ happens, then the above identity shows that both terms on the right are less than $\delta^2$. Since $1-\cos(2\pi  X_{j,r}/N)\ge 0$, this implies that for all $r$,
\[
1-\cos(2\pi  X_{j,r}/N)\le C\delta.
\]
Since $2\pi X_{j,r}/N\in (-\pi,\pi)$ and $\delta\in (0,1)$, a consequence of the above inequality is that
\begin{align}\label{cond1}
|X_{j,r}|\le C N\sqrt{\delta}.
\end{align}
This implies that
\[
|\sin(2\pi X_{j,r}/N) - 2\pi X_{j,r}/N|\le C\delta^{3/2}. 
\]
Thus,
\begin{align*}
\biggl|\sum_{r=1}^k p_r \sin(2\pi  X_{j,r}/N) - \sum_{r=1}^k p_r 2\pi X_{j,r}/N\biggr| &\le C \delta^{3/2}. 
\end{align*}
But, if $E_j$ happens, then 
\begin{align*}
\biggl|\sum_{r=1}^k p_r \sin(2\pi  X_{j,r}/N)\biggr| &\le \delta. 
\end{align*}
Combining, we get
\begin{align}\label{cond2}
\biggl|\sum_{r=1}^k p_r 2\pi X_{j,r}/N\biggr| \le C\delta. 
\end{align}
By \eqref{cond1} and \eqref{cond2}, 
\begin{align*}
\P(E_j) &\le \P\biggl(\biggl|\sum_{r=1}^k p_r 2\pi X_{j,r}\biggr| \le CN\delta \text{ and } |X_{j,r}|\le C N\sqrt{\delta} \text{ for all } r\biggr) \\
&\le  \P\biggl(\biggl|X_{j,1} +\frac{1}{p_1} \sum_{r=2}^k p_r X_{j,r}\biggr| \le CN\delta \text{ and } |X_{j,r}|\le C N\sqrt{\delta} \text{ for } 2\le r\le k \biggr).
\end{align*}
By the assumption that $N$ is prime, it follows that $X_{j,r}$ is uniformly distributed in $\{-(N-1)/2,\ldots,(N-1)/2\}$. Moreover, $X_{j,1},\ldots,X_{j,k}$ are independent. Thus, the above probability is bounded above by $C\delta \cdot \delta^{(k-1)/2} = C\delta^{(k+1)/2}$. Therefore,
\begin{align*}
\P(\tau \ge \delta^{-1}) &= \P(E_1\cup\cdots \cup  E_{N-1}) \le \sum_{j=1}^{N-1} \P(E_j) \le CN \delta^{(k+1)/2}. 
\end{align*}
Substituting $\delta = L^{-1}N^{-2/(k+1)}$ completes the proof.

\subsection{Proof of Theorem \ref{pathex}}\label{pathexproof}
Throughout this proof, $C, C_1, C_2,\ldots$ will denote positive constants that depend only on $d$ and the $p_i$'s, whose values may change from line to line. We will use Theorem \ref{paththm} to get the upper bound on the relaxation time. Let $E$ be as in Theorem \ref{paththm}. Note that an edge $(x,x\pm e_i)$ is in $E$ if and only if $p_{\pm i} >0$. Since at least one of $p_i$ and $p_{-i}$ is nonzero for each $1\le i\le d$, this allows us to construct a directed path $\Gamma_{x,y}$ connecting any $x\in (\Z/N\Z)^d$ to any other $y\in (\Z/N\Z)^d$ along the edges of $E$ such that $|\Gamma_{x,y}|\le d N$ for any $x$ and $y$. Thus, for any $e\in E$,
\[
\sum_{x,y:\, \Gamma_{x,y}\ni e} \mu(x)\mu(y)|\Gamma_{x,y}|\le d N. 
\]
Since 
\[
Q(e) \ge \frac{1}{N} \min\{p_i: -d\le i\le d, \, p_i\ne 0\}, 
\]
this shows that $B\le CN^2$, where $B$ is as in Theorem \ref{paththm}. Thus, by Theorem \ref{paththm}, $\tau \le CN^2$.

Next, let us prove the required lower bound on $\tau$. If $p_j=p_{-j}$ for some $1\le j\le d$, define $f:(\Z/N\Z)^d\to \C$ as $f(x) := e^{2\pi\I x_j/N}$. Then note that $\mu f = 0$, $\|f\|=1$ and 
\begin{align*}
Pf(x) &= (1-2p_j) f(x) + p_j(e^{2\pi\I (x_j+1)/N} + e^{2\pi\I (x_j-1)/N}) \\
&= (1-2p_j + 2p_j\cos(2\pi/N))f(x),
\end{align*}
which implies that $\|(I-P)f\| = O(N^{-2})$. Thus, $\tau\ge CN^2$. Next, suppose that  $p_j\ne p_{-j}$ for all $1\le j\le d$ and $p_j-p_{-j}\in \mathbb{Q}$ for at least two $j$'s. Then there exist integers $m_1,\ldots,m_d$, not all zero, such that 
\begin{align}\label{mipi}
\sum_{j=1}^d m_j(p_j-p_{-j}) = 0.
\end{align}
Define
\begin{align}\label{fdef}
f(x) := \exp\biggl(\frac{2\pi \I}{N} \sum_{j=1}^d m_j x_j\biggr),
\end{align}
so that $\mu f = 0$ (since $m_j$ are not all zero) and $\|f\|=1$.
Then 
\begin{align}\label{pfform}
Pf (x) &= p_0 f(x) + \sum_{j=1}^d p_j f(x+e_j) + \sum_{j=1}^d p_{-j} f(x-e_j)\notag \\
&= \biggl(p_0 + \sum_{j=1}^d (p_j e^{2\pi \I m_j/N} + p_{-j} e^{-2\pi \I m_j/N}) \biggr)f(x).
\end{align}
Note that by \eqref{mipi},
\begin{align*}
&p_0 + \sum_{j=1}^d (p_j e^{2\pi \I m_j/N} + p_{-j} e^{-2\pi \I m_j/N}) \\
&=  p_0 + \sum_{j=1}^d \biggl[p_j \biggl(1+ \frac{2\pi \I m_j}{N} + O(N^{-2})\biggr) + p_{-j} \biggl(1- \frac{2\pi \I m_j}{N} + O(N^{-2})\biggr)\biggr]\\
&= 1 + \frac{2\pi\I}{N}\sum_{j=1}^d(p_j-p_{-j}) m_j + O(N^{-2}) = 1+O(N^{-2}). 
\end{align*}
Thus, $\|f - Pf\|=O(N^{-2})$, and hence, $\tau \ge CN^2$.

\subsection{Proof of Theorem \ref{genpithm}}\label{genpiproof}
We need the following lemma. The proof is based on a standard idea from Diophantine approximation using the pigeon hole principle.
\begin{lmm}\label{diolmm}
For any $n$, any $x_1,\ldots,x_n\in [-1,1]$ and any integer $L\ge 1$, there exist integers $k_1,\ldots,k_n\in [-2L, 2L]$, not all zero, such that $|k_1x_1+\cdots+k_nx_n|\le (2L)^{-n+1}$. 
\end{lmm}
\begin{proof}
Consider the numbers $a_1x_1+\cdots +a_nx_n$ where $a_1,\ldots,a_n$ range over all integers between $-L$ and $L$. These numbers are all in $[-L,L]$. There are $(2L+1)^n$ choices of $(a_1,\ldots,a_n)$, and $[-L,L]$ can be partitioned into $(2L)^n$ sub-intervals of length $(2L)^{-n+1}$. Therefore, by the pigeon hole principle, there exist distinct $(a_1,\ldots,a_n)$ and $(b_1,\ldots,b_n)$ such that $a_1x_1+\cdots+a_n x_n$ and $b_1x_1+\cdots+b_nx_n$ belong to the same sub-interval. Taking $k_i=a_i-b_i$ completes the proof.
\end{proof}

We now prove the claimed lower bound on the relaxation time using this lemma. By Lemma~\ref{diolmm}, given any integer $L\ge 1$, there exist integers $m_1,\ldots,m_d\in [-2L, 2L]$, not all zero, such that
\begin{align}\label{mchoice}
\biggl|\sum_{i=1}^d m_i(p_i-p_{-i}) \biggr|\le (2L)^{-d+1}. 
\end{align}
Fixing $L$ and choosing $m_1,\ldots,m_d$ as above, define $f:(\Z/N\Z)^d \to \C$ exactly as in \eqref{fdef}.  The formula \eqref{pfform} remains valid. Note that by \eqref{mchoice},
\begin{align*}
&p_0 + \sum_{j=1}^d (p_j e^{2\pi \I m_j/N} + p_{-j} e^{-2\pi \I m_j/N}) \\
&=  p_0 + \sum_{j=1}^d \biggl[p_j \biggl(1+ \frac{2\pi \I m_j}{N} + O(L^2N^{-2})\biggr) + p_{-j} \biggl(1- \frac{2\pi \I m_j}{N} + O(L^2N^{-2})\biggr)\biggr]\\
&= 1 + \frac{2\pi\I}{N}\sum_{j=1}^d(p_j-p_{-j}) m_j + O(L^2N^{-2}) \\
&=  1 + O(L^{-d+1}N^{-1}) + O(L^2 N^{-2}). 
\end{align*}
Choosing $L$ of order $N^{1/(d+1)}$ gives $\|(I-P)f\|=O(N^{-2d/(d+1)})$, completing the proof of the lower bound on the relaxation time.

Let us now prove the upper bound. We need the following lemmas.
\begin{lmm}\label{prandlmm}
Choose $(p_i)_{-d\le i\le d}$ uniformly at random from the unit simplex in $\R^{2d+1}$. Then for any $m_{1},\ldots, m_{d}\in \Z$ and any $\ve>0$, 
\[
\P(|m_1(p_1-p_{-1}) + \cdots +m_d(p_d - p_{-d})|\le \ve) \le \frac{C\ve}{\max_{1\le i\le d}|m_i|},
\]
where $C$ depends only on $d$.
\end{lmm}
\begin{proof}
It is not hard to show that the random vector $(p_{i} - p_{-i})_{1\le i\le d}$ has a bounded probability density with respect to Lebesgue measure on $\R^{d}$. Also, it is not hard to see that given $m_1,\ldots,m_d\in \Z$, the set 
\[
\{(x_i)_{1\le i\le d}: |x_i|\le 1 \text{ for all } i \text{ and } |m_1x_1+\cdots+m_d x_d|\le \ve\}
\]
has volume at most $C\ve /\max_{1\le i\le d}|m_i|$. The claim follows from these two observations.
\end{proof}
\begin{lmm}\label{bcconseqlmm}
Let $(p_i)_{-d\le i\le d}$ be as in Lemma \ref{prandlmm}. Then almost surely, there is some $C>0$ such that  for all $m_1,\ldots, m_d\in \Z$ that are not all zero,
\begin{align}\label{mproperty}
|m_1(p_1-p_{-1}) + \cdots +m_d(p_d - p_{-d})| > C\psi(\max_{1\le i\le d}|m_i|),
\end{align}
where $\psi(x):= x^{-d+1}(\log (1+x))^{-2}$. 
\end{lmm}
\begin{proof}
For $m_1,\ldots,m_d\in \Z$ such that not all are zero, define the event
\[
E_{m_1,\ldots,m_d} := \{ |m_1(p_1-p_{-1}) + \cdots +m_d(p_d - p_{-d})| \le \psi(\max_{1\le i\le d} |m_i|)\},
\]
Let $\phi(x):= x^{-1}\psi(x) = x^{-d}(\log (1+x))^{-2}$. Then by Lemma \ref{prandlmm},
\begin{align*}
\sum_{(m_1,\ldots,m_d)\in \Z^d\setminus\{(0,\ldots,0)\}} \P(E_{m_1,\ldots,m_d}) &\le C\sum_{(m_1,\ldots,m_d)\in \Z^d\setminus\{(0,\ldots,0)\}} \phi(\max_{1\le i\le d} |m_i|)
\end{align*}
for some constant $C$ that depends only on $d$. It is easy to see that the sum on the right is finite. Therefore, by the Borel--Cantelli lemma, the events $E_{m_1,\ldots,m_d}$ can happen for only finitely many $(m_1,\ldots,m_d)$ with probability one. Thus, by choosing $C$ sufficiently small in the statement of the lemma (depending on the particular realization of the $p_i$'s), we arrive at the desired conclusion.
\end{proof}

We are now ready to complete the proof of the upper bound in Theorem~\ref{genpithm}. Note that by Lemma~\ref{bcconseqlmm}, almost all $(p_i)_{-d\le i\le d}$ satisfy the condition stated in the lemma. Take any such $(p_i)_{-d\le i\le d}$, and additionally assume that the $p_i$'s are all nonzero. Throughout the remainder of this proof, $C,C_1,C_2,\ldots$ will denote positive constants that depend only on $d$ and the $p_i$'s, whose values may change from line to line.

It is easy to check that the transition matrix for this problem is a normal matrix with respect to the ordinary inner product, and that its eigenvalues are 
\begin{align*}
\lambda_{m_1,\ldots,m_d} := p_0 + \sum_{j=1}^d (p_j e^{2\pi\I m_j/N} + p_{-j} e^{-2\pi\I m_j/N}),
\end{align*}
where $m_1,\ldots,m_d$ range in $\{0,1,\ldots,N-1\}$. By Theorem \ref{revnormthm}, the spectral gap of the chain (according to our definition) is given by
\begin{align*}
\gamma = \min_{\substack{0\le m_1,\ldots,m_d\le N-1,\\ \text{at least one $m_i\ne 0$}}}|1 - \lambda_{m_1,\ldots,m_d}|.
\end{align*}
Since $e^{\pm 2\pi\I m_j/N} = e^{\pm 2\pi\I (N-m_j)/N}$, we get the same value of $\gamma$ if we take  minimum over $0\le m_1,\ldots,m_d\le N/2$, not all zero. Next, let $\delta\in (0,1/2)$ be a number that we will choose later. Note that $\lambda_{m_1,\ldots,m_d}$ is a weighted average of points on the unit circle. Since the $p_i$'s are all nonzero, this shows that if some $m_j$ is bigger than $\delta N$ but less than or equal to $N/2$, then $|1-\lambda_{m_1,\ldots,m_d}|\ge C(\delta)$, where $C(\delta)$ is a positive constant that depends only on $\delta$, $d$, and the $p_i$'s. Thus, to show that $\gamma \ge N^{-2d/(d+1)+o(1)}$ as $N\to\infty$, it is sufficient to show that $\gamma' \ge N^{-2d/(d+1)+o(1)}$, where 
\[
\gamma' := \min_{\substack{0\le m_1,\ldots,m_d\le \delta N,\\ \text{at least one $m_i\ne 0$}}}|1 - \lambda_{m_1,\ldots,m_d}|.
\]
Take any $m_1,\ldots,m_d$ as in the above minimum. Observe that 
\begin{align*}
1 - \lambda_{m_1,\ldots,m_d}&= -\frac{2\pi \I}{N}\sum_{j=1}^d (p_j - p_{-j})m_j + \frac{2\pi^2}{N^2} \sum_{j=1}^d (p_j +p_{-j})m_j^2 + R_{m_1,\ldots,m_d}, 
\end{align*}
where 
\[
|R_{m_1,\ldots,m_d}| \le\frac{(2\pi)^3}{6N^3} \sum_{j=1}^d(p_j+p_{-j}) |m_j|^3\le \frac{(2\pi)^3\delta}{6N^2} \sum_{j=1}^d(p_j+p_{-j}) m_j^2,
\]
because the absolute value of the third derivative of $x \mapsto e^{\I x}$ is bounded by $1$ everywhere on the real line. Thus,
\begin{align}\label{mainlambda}
|1 - \lambda_{m_1,\ldots,m_d}| &\ge \biggl|-\frac{2\pi \I}{N}\sum_{j=1}^d (p_j - p_{-j})m_j + \frac{2\pi^2}{N^2} \sum_{j=1}^d (p_j +p_{-j})m_j^2\biggr| - |R_{m_1,\ldots,m_d}|\notag \\
&= \biggl[\biggl(\frac{2\pi}{N}\sum_{j=1}^d (p_j - p_{-j})m_j\biggr)^2 + \biggl(\frac{2\pi^2}{N^2} \sum_{j=1}^d (p_j +p_{-j})m_j^2\biggr)^2\biggr]^{1/2}\notag \\
&\qquad \qquad  - \frac{(2\pi)^3\delta}{6N^2} \sum_{j=1}^d(p_j+p_{-j}) m_j^2. 
\end{align}
Now, if 
\begin{align}\label{longineq}
\biggl|\frac{2\pi}{N}\sum_{j=1}^d (p_j - p_{-j})m_j\biggr| \ge \biggl|\frac{2\pi^2}{N^2} \sum_{j=1}^d (p_j +p_{-j})m_j^2\biggr|, 
\end{align}
then \eqref{mainlambda} shows that 
\begin{align*}
|1 - \lambda_{m_1,\ldots,m_d}| &\ge \biggl|\frac{2\pi}{N}\sum_{j=1}^d (p_j - p_{-j})m_j\biggr| - \frac{(2\pi)^3\delta}{6N^2} \sum_{j=1}^d(p_j+p_{-j}) m_j^2\\
&\ge \biggl|\frac{\pi}{N}\sum_{j=1}^d (p_j - p_{-j})m_j\biggr| + \frac{C(1-\delta)}{N^2} \sum_{j=1}^d(p_j+p_{-j}) m_j^2,
\end{align*}
provided that $\delta$ is chosen to be smaller than a suitable universal constant. On the other hand, if the reverse inequality holds in \eqref{longineq}, then
\begin{align*}
|1 - \lambda_{m_1,\ldots,m_d}| &\ge  \biggl|\frac{2\pi^2}{N^2} \sum_{j=1}^d (p_j +p_{-j})m_j^2\biggr| - \frac{(2\pi)^3\delta}{6N^2} \sum_{j=1}^d(p_j+p_{-j}) m_j^2\\
&\ge \frac{C(1-\delta)}{N^2} \sum_{j=1}^d(p_j+p_{-j}) m_j^2\\
&\ge  \biggl|\frac{C_1}{N}\sum_{j=1}^d (p_j - p_{-j})m_j\biggr| + \frac{C_2}{N^2} \sum_{j=1}^d(p_j+p_{-j}) m_j^2,
\end{align*}
again provided that $\delta$ is sufficiently small. Thus, with a small enough choice of $\delta$, we have 
\begin{align}\label{rweig}
\gamma' \ge \min_{\substack{0\le m_1,\ldots,m_d\le \delta N,\\ \text{at least one $m_i\ne 0$}}} \biggl( \biggl|\frac{C_1}{N}\sum_{j=1}^d (p_j - p_{-j})m_j\biggr| + \frac{C_2}{N^2} \sum_{j=1}^d(p_j+p_{-j}) m_j^2\biggr). 
\end{align}
Take any $m_1,\ldots,m_d$ as in the above minimum. Let $L := \max_{1\le i\le d} |m_i|$. Then by \eqref{mproperty}, 
\begin{align*}
\biggl|\frac{C_1}{N}\sum_{j=1}^d (p_j - p_{-j})m_j\biggr| + \frac{C_2}{N^2} \sum_{j=1}^d(p_j+p_{-j}) m_j^2&\ge \frac{C_3 \psi(L)}{N} + \frac{C_4 L^2}{N^2}.
\end{align*}
Using the fact that $1\le L\le \delta N$, it is not hard to verify that right side is bounded below by $N^{-2d/(d+1) + o(1)}$, irrespective of the value of $L$. This completes the proof.

\subsection{Proof of Theorem \ref{explicitthm}}\label{explicitproof}
In this proof, $C, C_1, C_2,\ldots$ will denote positive constants that depend only on $\alpha$, whose values may change from line to line. The required upper bound on the spectral gap is already available from Theorem \ref{genpithm}. So we only have to prove the required lower bound. By the inequality \eqref{rweig}, it suffices to show that 
\begin{align}\label{lasttoshow}
\min_{\substack{0\le k,m\le N-1,\\k\ne 0 \text{ or } m\ne 0}} \biggl(\biggl|\frac{1}{N}(\alpha k + (1-\alpha)m)\biggr| + \frac{1}{N^2}(\alpha k^2 + (1-\alpha) m^2)\biggr) \ge CN^{-4/3},
\end{align}
where $C$ is a positive constant that does not depend on $N$. Since $\alpha$ solves a quadratic equation with integer coefficients, we can write
\[
\alpha = \frac{a + s\sqrt{b}}{c}
\]
where $a,b,c\in \Z$, $s\in\{-1,1\}$, $b\ge 0$, and $c\ne 0$. Since $\alpha$ is irrational, $b$ is not a perfect square. Take any two integers $p,q$, at least one of which is nonzero. Since $b$ is not a perfect square and $p,q$ are integers which are not both zero, we must have $|q^2 b - p^2|\ge 1$.
But
\begin{align*}
|q^2b - p^2| &= |q\sqrt{b} - p||q\sqrt{b}+p| \le C |q \sqrt{b} - p| (|p|+|q|). 
\end{align*}
Thus, 
\begin{align}\label{pqlower}
|q \sqrt{b}-p|\ge \frac{C}{|p|+|q|}.
\end{align}
Now, for any $k,m\in \Z$, 
\begin{align*}
|\alpha k + (1-\alpha)m| &= \biggl|\frac{a+s\sqrt{b}}{c} (k-m) + m\biggr|\\
&= \frac{|s(k-m)\sqrt{b} + a(k-m) + mc|}{|c|}.
\end{align*}
If at least one of $k,m$ is nonzero, then it is easy to see that at least one of $s(k-m)$ and $a(k-m)+mc$ is nonzero. Thus, by \eqref{pqlower}, we get
\[
|\alpha k + (1-\alpha)m| \ge \frac{C_1}{|k-m| + |m|} \ge \frac{C_2}{|k|+|m|}.
\]
Thus, we have
\begin{align*}
&\biggl|\frac{1}{N}(\alpha k + (1-\alpha)m)\biggr| + \frac{1}{N^2}(\alpha k^2 + (1-\alpha) m^2) \\
&\ge \frac{C_1}{N(|k|+|m|)} + \frac{C_2(|k|+|m|)^2}{N^2}\\
&\ge C_3 \biggl(\frac{1}{N(|k|+|m|)}\biggr)^{2/3} \biggl(\frac{(|k|+|m|)^2}{N^2}\biggr)^{1/3} = \frac{C_3}{N^{4/3}},
\end{align*}
where the second step follows by Young's inequality. This proves \eqref{lasttoshow}, completing the proof of the theorem.

\subsection{Proof of Theorem \ref{explicitthm2}}\label{explicit2proof}
In this proof, $C, C_1, C_2,\ldots$ will denote positive constants that depend only on $\alpha$ and $\ve$, whose values may change from line to line. 
As in the proof of Theorem \ref{explicitthm}, it suffices to show that 
\begin{align*}
\min_{\substack{0\le k,m\le N-1,\\k\ne 0 \text{ or } m\ne 0}} \biggl(\biggl|\frac{1}{N}(\alpha k + (1-\alpha)m)\biggr| + \frac{1}{N^2}(\alpha k^2 + (1-\alpha) m^2)\biggr) \ge CN^{-4/3-\ve}.
\end{align*}
Without loss of generality, let us assume that $\ve \in (0,2/3)$. Let $\ve'$ solve
\[
\frac{4+2\ve'}{3+\ve'} = \frac{4}{3}+\ve.
\]
Using the assumption that $\ve \in (0,2/3)$, is easy to check that $\ve' >0$. Thus, by Roth's theorem about rational approximations of algebraic numbers~\cite{roth55}, there is a constant $C$ depending only on $\alpha$ and $\ve$ such that any two coprime integers $p$ and $q$ with $q\ne0$, we have 
\[
\biggl|\alpha - \frac{p}{q}\biggr| > \frac{C}{|q|^{2+\ve'}}. 
\]
This shows that if $0\le k,m\le N-1$ with $k\ne m$, then 
\begin{align*}
\biggl|\alpha - \frac{-m}{k-m}\biggr| > \frac{C}{|k-m|^{2+\ve'}}. 
\end{align*}
This can be rewritten as 
\begin{align*}
|\alpha k + (1-\alpha)m| &> \frac{C}{|k-m|^{1+\ve'}}\ge \frac{C_1}{(|k|+|m|)^{1+\ve'}}. 
\end{align*}
Note that we proved the above inequality under the assumption that $k\ne m$. But it is easy to see that it holds even if $k=m> 0$. Thus, by Young's inequality, we have that whenever $0\le k,m\le N-1$ and at least one of $k,m$ is nonzero,
\begin{align*}
&\biggl|\frac{1}{N}(\alpha k + (1-\alpha)m)\biggr| + \frac{1}{N^2}(\alpha k^2 + (1-\alpha) m^2) \\
&\ge \frac{C_1}{N(|k|+|m|)^{1+\ve'}} + \frac{C_2(|k|+|m|)^2}{N^2}\\
&\ge C_3 \biggl(\frac{1}{N(|k|+|m|)^{1+\ve'}}\biggr)^{2/(3+\ve')} \biggl(\frac{(|k|+|m|)^2}{N^2}\biggr)^{(1+\ve')/(3+\ve')} \\
&= \frac{C_3}{N^{(4+2\ve')/(3+\ve')}} = C_3 N^{-4/3 - \ve}.
\end{align*}
This completes the proof.

\subsection{Proof of Theorem \ref{cdgthm}}\label{cdgproof}
Note that for any $n$,
\begin{align*}
X_n &= 2^n X_0 + 2^{n-1}\ve_1+\cdots +2\ve_{n-1} +\ve_n \bmod N. 
\end{align*}
Thus, for any $0\le j\le N-1$, 
\begin{align}\label{fourier}
\E(e^{2\pi \I j X_n/N}|X_0=x) &= e^{2\pi\I j 2^nx/N}\prod_{k=1}^n \E(e^{2\pi \I j 2^{n-k}\ve_k/N})\notag \\
&= e^{2^{n+1}\pi\I j x/N} \prod_{k=1}^n \frac{1+2\cos(2^{n-k+1}\pi j/N)}{3} \notag \\
&= e^{2^{n+1}\pi\I jx/N} \prod_{k=1}^{n} \frac{1+2\cos(2^k\pi j/N)}{3}. 
\end{align}
Now, the functions $x\mapsto e^{2\pi \I jx/N}$, for $j=0,1,\ldots,N-1$, form an orthonormal basis of $L^2(\mu)$, where $\mu$ is the uniform distribution on $\Z/N\Z$. Thus, any $f:\ms \to \C$ with $\mu f=0$   can be expressed as
\[
f(x) = \sum_{j=1}^{N-1} \alpha_j e^{2\pi \I jx/N}
\]
for some coefficients $\alpha_1,\ldots,\alpha_{N-1}$ such that
\[
\sum_{j=1}^{N-1} |\alpha_j|^2 = \|f\|^2,
\]
where $\|\cdot\|$ denote the norm induced by the inner product \eqref{ipdef}. Using this and \eqref{fourier}, we get
\begin{align}\label{pnf}
P^n f(x) &= \sum_{j=1}^{N-1}\alpha_j  e^{2^{n+1}\pi\I jx/N}\biggl( \prod_{k=1}^{n} \frac{1+2\cos(2^k\pi j/N)}{3}\biggr). 
\end{align}
Since $N$ is an odd prime, it follows that for any $1\le j\ne k\le N-1$, 
\[
2^{n+1}(j-k)\ne 0\bmod N,
\]
and thus, the functions $e^{2^{n+1}\pi\I jx/N}$ and $e^{2^{n+1}\pi\I kx/N}$ are orthogonal to each other. Therefore, by~\eqref{pnf}, 
\begin{align}\label{pnfourier}
\|P^n f \|^2 &= \sum_{j=1}^{N-1}|\alpha_j|^2 \biggl( \prod_{k=1}^{n} \frac{1+2\cos(2^k\pi j/N)}{3}\biggr)^2.
\end{align}
Now we need the following lemmas. Throughout, $x \bmod N$ will denote the remainder of an integer $x$ modulo $N$, as an element of $\{0,1,\ldots,N-1\}$. 
\begin{lmm}\label{periodlmm}
Take any $1\le j\le N-1$, and any $k\ge 1$ such that $2^k j \bmod N \in [1, N/3) \cup (2N/3, N-1)$. Then there exists some $1\le l \le \log_2 N$ such that $2^{k+l} j\bmod N \in [N/3, 2N/3]$, where $\log_2 N$ is the logarithm of $N$ in base $2$. 
\end{lmm}
\begin{proof}
Let $a:= 2^k j \bmod N$. Suppose that $a < N/3$. Let $l$ be the smallest positive integer such that $2^l a \ge N/3$. Clearly, $l\le \log_2 N$. Since $2^{l-1} a < N/3$, we have $2^la < 2N/3$. Thus, $2^{k+l} j \bmod N = 2^l a\bmod N\in [N/3, 2N/3]$. Next suppose that $a> 2N/3$. Let $b := N-a$, so that $b\in [1, N/3)$. By the previous step, there is some $l\le \log_2N$ such that $2^l b\bmod N \in [N/3, 2N/3]$. This implies that $N - (2^l b\mod N) \in [N/3, 2N/3]$. But
\begin{align*}
N - (2^l b\bmod N) \equiv -2^l b \bmod N \equiv 2^l a \bmod N.
\end{align*}
Thus, $2^{k+l} j \bmod N\in [N/3, 2N/3]$. 
\end{proof}
\begin{lmm}\label{countlmm}
Take any $1\le j\le N-1$ and any $n\ge 1$. There are at least $\lfloor n/(2+\log_2 N) \rfloor$ elements $k\in \{1,\ldots,n\}$ such that $2^k j \bmod N \in [N/3,2N/3]$.
\end{lmm}
\begin{proof}
Let $A$ be the set of all $1\le k\le n$ such that $2^k j \bmod N \in [N/3, 2N/3]$. Divide the interval $[1,n]$ into closed sub-intervals of length $1+\log_2 N$, with the last one possibly shorter. Let $I$ be such a sub-interval other than the last one, and let $I^\circ$ denote its interior. Let $k$ be the smallest integer such that $k\in I^\circ$. Then either $k\in A$, or by Lemma \ref{periodlmm}, $k+l\in A$ for some $l\le \log_2 N$. In the latter case, note that $k+l\in I^\circ$, since $I$ has length $2+\log_2 N$. Thus, the interior of each  sub-interval, except possibly the last one, must contain at least one $k\in A$. Since the interiors are disjoint, this completes the proof.
\end{proof}
\begin{lmm}\label{coslmm}
For any $1\le j\le N-1$ and any $n\ge 1$,
\begin{align*}
\biggl| \prod_{k=1}^{n} \frac{1+2\cos(2^k\pi j/N)}{3}\biggr| &\le (2/3)^{\lfloor n/(2+\log_2 N) \rfloor}. 
\end{align*}
\end{lmm}
\begin{proof}
Note that $| \cos(2^k\pi j/N)|\le 1$ for all $k$, and if $2^k j\bmod N \in [N/3, 2N/3]$, then 
\[
\biggl|\frac{1+2 \cos(2^k\pi j/N)}{3}\biggr| \le\biggl|\frac{1+2\cos(\pi/3)}{3}\biggr| = \frac{2}{3}.
\]
By Lemma \ref{countlmm}, this completes the proof.
\end{proof}

Let us now return to the proof of Theorem \ref{cdgthm}. By \eqref{pnfourier} and Lemma \ref{coslmm},
\begin{align}\label{pnbound}
\|P^n f \|^2 &\le (2/3)^{2\lfloor n/(2+\log_2 N) \rfloor}\sum_{j=1}^{N-1} |\alpha_j|^2 = (2/3)^{2\lfloor n/(2+\log_2 N) \rfloor}\|f\|^2.
\end{align}
Now take any $h:\ms \to \R$ with $\mu h = 0$ and $\|h\|=1$. Let $f:= h-Ph$. Then note that $\mu f = 0$. Thus, by \eqref{pnbound}, the function
\begin{align*}
w := \sum_{n=0}^\infty P^n f
\end{align*}
is well-defined. We claim that $w = h$. Indeed, by \eqref{pnbound}, we get
\begin{align*}
w =\sum_{n=0}^\infty P^n (h-Ph) = \sum_{n=0}^\infty P^n h - \sum_{n=0}^\infty P^{n+1} h = h. 
\end{align*}
Thus,  by \eqref{pnbound}, 
\[
\|h\| \le \sum_{n=0}^\infty \|P^n f\| \le C\|f\|\log N,
\]
where $C$ is a positive universal constant. But we know that $\|h\|=1$. This proves that the spectral gap of the Chung--Diaconis--Graham chain is at least of order $1/\log N$. 

Next, we prove the required upper bound on the spectral gap. Let $k$ be the largest even integer such that $2^k < N/8$. For $0\le j\le k$, let 
\[
a_j := 
\begin{cases}
j (\log N)^{-3/2} &\text{ if } j\le k/2,\\
(k-j)(\log N)^{-3/2} &\text{ if } j> k/2,
\end{cases}
\]
and for $0\le j\le k+1$, let
\[
b_j :=  \prod_{l=0}^j\frac{1+2\cos(2^{l}\pi/N)}{3}. 
\]
Define a function $h:\Z/N\Z \to \C$ as
\begin{align*}
h(x) &:= \sum_{j=0}^ka_j b_j e^{2^{j+1}\pi \I x/N}. 
\end{align*}
Then note that $\mu h = 0$, and 
\begin{align*}
Ph(x) &= \E(h(X_1)|X_0=x)\\
&= \sum_{j=0}^k a_j b_j \E(e^{2^{j+1}\pi \I (2x+\ve_1)/N})\\
&= \sum_{j=0}^k a_j b_j e^{2^{j+2}\pi \I x/N}\frac{1+2 \cos(2^{j+1}\pi/N)}{3}\\
&= \sum_{j=0}^k a_j b_{j+1}   e^{2^{j+2}\pi \I x/N} = \sum_{j=1}^{k+1} a_{j-1} b_j e^{2^{j+1}\pi \I x/N}. 
\end{align*}
Thus, we have 
\begin{align*}
h(x)-Ph(x) &= a_0 b_0 e^{2\pi \I x/N} + \sum_{j=1}^k (a_j - a_{j-1}) b_j e^{2^{j+1}\pi \I x/N} - a_k b_{k+1}e^{2^{k+2}\pi \I x/N}.
\end{align*}
Since $2^k < N/8$, the numbers $2,4,8,\ldots, 2^{k+2}$ are distinct modulo $N$. Thus, the above expression gives 
\begin{align*}
\|h-Ph\|^2 &= a_0^2b_0^2 + \sum_{j=1}^k (a_j-a_{j-1})^2 b_j^2 + a_k^2b_{k+1}^2\\
&= \sum_{j=1}^k (a_j-a_{j-1})^2b_j^2 \le \sum_{j=1}^k (a_j - a_{j-1})^2 \le k (\log N)^{-3} \le C(\log N)^{-2},
\end{align*}
where $C$ is a universal constant. Next, using the fact that $\cos x \ge 1-\frac{1}{2}x^2$ for all $x\in [-1/4, 1/4]$, we have that for all $0\le j\le k$,
\begin{align*}
b_j &\ge \prod_{l=0}^j \biggl(1 - \frac{2^{2l}\pi^2}{3N^2}\biggr)\\
&\ge 1 - \sum_{l=0}^j \frac{2^{2l}\pi^2}{3N^2}\ge 1 - \frac{2^{2j+1}\pi^2}{3N^2}\ge 1-\frac{\pi^2}{64}\ge \frac{3}{4}. 
\end{align*}
Thus, we have
\begin{align*}
\|h\|^2 &= \sum_{j=0}^k a_j^2 b_j^2 \ge \frac{k^2}{16(\log N)^3}\sum_{k/4\le j\le 3k/4} b_j^2 \ge \frac{3k^3}{128(\log N)^3}. 
\end{align*}
Since $k$ is of order $\log N$, this shows that $\|h\|$ is bounded below by a positive constant that does not depend on $N$. Thus, the spectral gap of the chain is bounded above by
\[
\frac{\|(I-P)h\|}{\|h\|} \le C (\log N)^{-1},
\]
where $C>0$ does not depend on $N$.

\subsection{Proof of Theorem \ref{cardthm}}\label{cardproof}
By \cite[Example 2a]{diaconissaloffcoste96}, the second-largest singular value of $P$ is $\le 1- 1/(41N^3)$. By Theorem~\ref{revthm}, this shows that $\tau \le 41N^3$. 

For the lower bound, let us first simplify the problem as follows. Let $Q$ be the transition matrix of the Markov chain which swaps the top two cards with probability $1/2$, and takes out the bottom card and places it on top with probability $1/2$. Then $P = (I + 2Q)/3$. Thus, $I-P = 2(I-Q)/3$, which shows that if $\tau'$ is the relaxation time of this new Markov chain, then $\tau =3\tau'/2$. 

Now consider the movement of a given card (say, card number $1$) under the $Q$-chain. Let $Y_0$ be the starting location of the card, which we denote by a number between $0$ (top) and $N-1$ (bottom). Each time the card moves down one slot, or goes from the bottom of the deck to the top, we increase $Y_i$ to $Y_{i+1}=Y_i+1$. Each time it moves up one slot, we decrease $Y_{i+1}=Y_i-1$. Note that the top card can only move down, and so $Y_i\ge 0$ for all $i$. The actual location of the card in the deck at time $i$ is the remainder of $Y_i$ modulo $N$. Let us denote this by $Z_i$. 

Let $\xi_i$ be a random variable which is $0$ if the top two cards are swapped at time $i$, and $1$ if the bottom card is taken out and placed on top at time $i$. Then $\xi_0,\xi_1,\ldots$ are i.i.d.~Bernoulli$(1/2)$ random variables. Note that
\begin{align*}
Y_{i+1} &= 
\begin{cases}
Y_i + \xi_i &\text{ if } Z_i \ge 2,\\
Y_i -1 &\text{ if } Z_i = 1 \text{ and } \xi_i = 0,\\
Y_i+1 &\text{ if } Z_i = 1 \text{ and } \xi_i =1, \\
Y_i + 1 &\text{ if } Z_i = 0.
\end{cases}
\end{align*}
To put it more succinctly,
\begin{align*}
Y_{i+1} &= Y_i + \xi_i 1_{\{Z_i\ge 2\}} + (2\xi_i -1)1_{\{Z_i = 1\}} + 1_{\{Z_i = 0\}}\\
&= Y_i + \xi_i (1-1_{\{Z_i=1\}} - 1_{\{Z_i=0\}})  + (2\xi_i -1)1_{\{Z_i = 1\}} + 1_{\{Z_i = 0\}}\\
&= Y_i + \xi_i(1+1_{\{Z_i=1\}} - 1_{\{Z_i=0\}} ) + 1_{\{Z_i = 0\}} - 1_{\{Z_i = 1\}}.
\end{align*}
Let $W_i := 1_{\{Z_i=1\}} - 1_{\{Z_i=0\}}$. The above formula shows that for any~$n\ge 1$, 
\begin{align}\label{ynid0}
Y_n &= Y_0 + \sum_{i=0}^{n-1} (\xi_i(1+W_i) -W_i). 
\end{align}
Now consider a second card, and define $Y_i'$, $Z_i'$ and $W_i'$ analogously for this card. The above identity continues to hold, and therefore,
\begin{align}\label{ynid}
Y_n - Y_n' &= Y_0-Y_0' + \sum_{i=0}^{n-1} (\xi_i (W_i- W_i') -W_i+W_i')\notag \\
&= Y_0-Y_0' + \sum_{i=0}^{n-1} \eta_i(W_i- W_i') -\frac{1}{2}\sum_{i=0}^{n-1}(W_i-W_i'),
\end{align}
where $\eta_i := \xi_i - 1/2$.

Now, observe that if $Z_i=0$ for some $i$, then $Z_{i+1}$ must be $1$, because a card at position $0$ must move to position $1$ in the next step. Moreover, if $Z_i \ne 0$ for some $i$, then $Z_{i+1} \ne 1$, because a card can land up at position $1$ only if it was at position $0$ in the previous step. Thus, for any $i$, $1_{\{Z_{i+1}=1\}} = 1_{\{Z_i=0\}}$.  This shows that for any $n\ge 2$, 
\begin{align*}
\sum_{i=0}^{n-1} W_i &= \sum_{i=0}^{n-1}(1_{\{Z_i=1\}} - 1_{\{Z_i=0\}}) \\
&= 1_{\{Z_0=1\}} - 1_{\{Z_{n-1}=1\}} + \sum_{i=0}^{n-2}(1_{\{Z_{i+1}=1\}} - 1_{\{Z_i=0\}})\\
&= 1_{\{Z_0=1\}} - 1_{\{Z_{n-1}=1\}}.
\end{align*}
Using this in \eqref{ynid}, together with the analogous identity for $W_i$, we get 
\begin{align*}
Y_n - Y_n' &= Y_0-Y_0' + \sum_{i=0}^{n-1} \eta_i(W_i- W_i') \\
&\qquad \qquad -\frac{1}{2}(1_{\{Z_0=1\}} - 1_{\{Z_{n-1}=1\}} - 1_{\{Z_0'=1\}} + 1_{\{Z_{n-1}'=1\}}). 
\end{align*}
In particular,
\begin{align}\label{ynineq}
|(Y_n-Y_n') - (Y_0-Y_0')| &\le \biggl|\sum_{i=0}^{n-1} \eta_i(W_i- W_i')\biggr| + 1. 
\end{align}
We now make the crucial observation, using \eqref{ynid0} and induction, that for each $i$, $Y_i$ and $Y_i'$ are functions of $\xi_0,\ldots,\xi_{i-1}$ (fixing $Y_0$ and $Y_0'$). This shows that
\begin{align*}
\E\biggl|\sum_{i=0}^{n-1} \eta_i(W_i- W_i')\biggr|^2 &= \sum_{i=0}^{n-1} \E(\eta_i^2(W_i-W_i')^2)\\
&\le \frac{1}{2}\sum_{i=0}^{n-1} \E(W_i^2 + {W_i'}^2)\\
&= \frac{1}{2}\sum_{i=0}^{n-1} (\P(Z_i = 0 \text{ or } 1) + \P(Z_i = 0 \text{ or } 1)). 
\end{align*}
Now suppose that the initial location $Y_0$ is uniformly random over $0,\ldots,N-1$. Then so is each $Y_i$, and thus, $\P(Z_i = 0 \text{ or } 1) = 2/N$. The same holds for $Z_i'$. Thus, if $Y_0$ and $Y_0'$ are uniformly distributed over $0,\ldots, N-1$, then 
\begin{align*}
\E\biggl|\sum_{i=0}^{n-1} \eta_i(W_i- W_i')\biggr|^2 &\le \frac{2n}{N}. 
\end{align*}
Combining this with \eqref{ynineq}, we get
\begin{align}\label{ynmainineq}
\E|(Y_n-Y_n') - (Y_0-Y_0')|^2 &\le \frac{4n}{N} + 2. 
\end{align} 
Now, for $x\in S_N$, let $x(i)$ denote the position of card $i$ in the deck arranged according to the permutation $x$, where the positions are numbered from $0$ (top) to $N-1$ (bottom). Define
\[
f(x) := \cos(2\pi(x(1)-x(2))/N). 
\]
Let $X_0, X_1,\ldots$ be a stationary Markov chain on $S_N$ with transition matrix $Q$. Let $Y_i$ and $Y_i'$ be the locations of cards $1$ and $2$ at time $i$, as defined above. Then $Y_i = X_i(1)\bmod N$ and $Y_i' = X_i(2)\bmod N$.  Moreover, $Y_0$ and $Y_0'$ are uniformly distributed on $\{0,\ldots,N-1\}$, since $X_0$ is uniformly distributed on $S_N$. Thus, for any $n$, the inequality \eqref{ynmainineq} shows that 
\begin{align*}
\E|f(X_n)-f(X_0)|^2 &= \E|\cos(2\pi(X_n(1)-X_n(2))/N) - \cos(2\pi(X_0(1)-X_0(2))/N)|^2\\
&= \E|\cos(2\pi(Y_n - Y_n')/N) - \cos(2\pi(Y_0-Y_0')/N)|^2\\
&\le \frac{4}{\pi^2N^2}\E|(Y_n-Y_n') - (Y_0-Y_0')|^2\\
&\le \frac{4}{\pi^2N^2}\biggl(\frac{4n}{N} + 2\biggr). 
\end{align*}
Thus, for any $n$,
\begin{align*}
\E\biggl|\frac{1}{n}\sum_{i=0}^{n-1} f(X_i) - f(X_0)\biggr|^2 &\le \frac{1}{n}\sum_{i=0}^{n-1} \E|f(X_i)-f(X_0)|^2 \\
&\le \frac{4}{\pi^2N^2n}\sum_{i=0}^{n-1}\biggl(\frac{4i}{N} + 2\biggr)\\
&\le \frac{8n}{\pi^2 N^3} + \frac{8}{\pi^2 N^2}.
\end{align*}
On the other hand, by Theorem \ref{avgthm}, 
\begin{align*}
\E\biggl|\frac{1}{n}\sum_{i=0}^{n-1} f(X_i) - \E(f(X_0))\biggr|^2&\le \frac{4\tau'}{n}\var(f(X_0)). 
\end{align*}
Combining the above inequalities, we get
\begin{align*}
\var(f(X_0)) &\le \frac{8\tau'}{n}\var(f(X_0)) + \frac{16n}{\pi^2 N^3} + \frac{16}{\pi^2 N^2}.
\end{align*}
It is easy to see that $\var(f(X_0))$ is bounded below by a positive constant that does not depend on $N$, because $(X_0(1)-X_0(2))/N$ has a non-degenerate limiting distribution as $N\to \infty$. Thus, for any $n$, the right side of the above display is bounded below by a positive constant that does not depend on $N$ or $n$. Taking $n = cN^3$ for some sufficiently small $c$, we get the desired lower bound on $\tau'$, provided that $N$ is large enough. The extension to all $N$ follows from the observation $\tau' > 0$ for any $N$, which is a consequence of the irreducibility of the chain.

\section*{Acknowledgements}
I thank Noga Alon, Persi Diaconis, Jonathan Hermon, Hariharan Narayanan, Piyush Srivastava and the anonymous referee for many helpful comments and references. This research was partially supported by NSF grants DMS-2113242 and DMS-2153654.

\bibliographystyle{imsart-number}

\bibliography{myrefs}

\end{document}